\documentclass[reqno,12pt]{amsart}
\usepackage{amsmath,amsthm,amssymb}
\usepackage{tikz}
\usetikzlibrary{arrows}
\usetikzlibrary{backgrounds, calc, positioning}

\usepackage{amsmath,amsthm,amsfonts,amssymb,mathtools}
\usepackage[T2A]{fontenc}
\usepackage{cmap}
\usepackage{icomma}

\usepackage{euscript}
\usepackage{mathrsfs}

	\newcommand{\Mod}[1]{\ (\mathrm{mod}\ #1)}

\usepackage[colorinlistoftodos]{todonotes}
\usepackage[colorlinks=true, allcolors=blue]{hyperref}

\date{}

\newtheorem{theorem}{Theorem}[section]
\newtheorem{lemma}[theorem]{Lemma}
\newtheorem{corollary}[theorem]{Corollary}
\newtheorem{proposition}[theorem]{Proposition}
\newtheorem{example}[theorem]{Example}
\newtheorem{remark}[theorem]{Remark}
\newtheorem{definition}[theorem]{Definition}
\newtheorem{question}[theorem]{Question}

\newtheorem{property}[theorem]{Property}

\numberwithin{equation}{section}

\newcommand*\xor{\mathbin{\oplus}}
\newcommand*\eps {\varepsilon}

 \title{On multistochastic Monge--Kantorovich problem, bitwise operations,  and fractals}

\author{Nikita~A.~Gladkov}
\address{National Research University Higher School of Economics, Moscow,  Russia}
\email{gladkovna@gmail.com}

\author{Alexander~V.~Kolesnikov}
\address{National Research University Higher School of Economics, Moscow,  Russia}
\email{Sascha77@mail.ru}

\author{Alexander~P.~Zimin}
\address{National Research University Higher School of Economics, Moscow,  Russia}
\email{alekszm@gmail.com}

\thanks{
 The
second named author was supported by RFBR
 project 17-01-00662 and  DFG project RO 1195/12-1. The article was prepared
within the framework of the Academic Fund Program at the National Research University Higher
School of Economics (HSE) in 2017--2018 (grant No 17-01-0102) and by the Russian Academic
Excellence Project 5-100.
}


\begin{document}

  \begin{abstract}
  The multistochastic  $ (n,k)$-Monge--Kantorovich problem  on a product space $\prod_{i=1}^n X_i$
  is an extension of the classical Monge--Kantorovich problem. This problem is considered on the space of measures  with fixed projections
  onto   $X_{i_1} \times \ldots \times X_{i_k}$ for all $k$-tuples $\{i_1, \ldots, i_k\} \subset \{1, \ldots, n\}$ for a given $1 \le k < n$.
  In our paper we study well-posedness of the primal and the corresponding dual problem. Our central result 
  describes a solution $\pi$ to the following important model case: $n=3, k=2, X_i = [0,1]$, the cost function $c(x,y,z) = xyz$, and the corresponding two--dimensional   projections are Lebesgue measures on $[0,1]^2$. We prove, in particular, that 
   the mapping $(x,y) \to x \oplus y$, where $\oplus$ is the bitwise addition (xor- or Nim-addition) on $[0,1] \cong \mathbb{Z}_2^{\infty}$, is the corresponding optimal transportation. In particular, the support of $\pi$ is the Sierpi\'nski tetrahedron. 
   In addition, we describe a solution to the corresponding dual problem. 
  \end{abstract}

\maketitle

\section{Introduction}

In this paper we consider a natural modification of the Monge--Kantorovich mass transportation problem
which we call ``multistochastic Monge--Kantorovich problem''. To our best knowledge, this  problem has  never been studied before.

Assume we are given two probability measures $\mu, \nu$ on measurable spaces $X, Y$ and a function
$c \colon X \times Y \to \mathbb{R}$.
Let us remind the reader that the classical Kantorovich or transportation problem
is a problem of minimization of the functional
\begin{equation}
\label{KF}
\int_{X \times Y} c(x,y) d \pi,
\end{equation}
on the set $\Pi(\mu,\nu)$ of probability measures on $X \times Y$ with fixed marginals $\mu, \nu$.

An important tool to attack this problem coming from the linear programming theory is the so-called dual  transportation problem:
maximize
$$
\int f d \mu + \int g d\nu
$$
on the set of couples of (integrable) functions $(f,g)$ satisfying
$f(x) + g(y) \le c(x,y)$.

The most classical cost function $c$ is given by the distance function, but the quadratic cost function $c(x,y) = |x-y|^2$ gains incredible popularity because of impressive number of applications.
For the quadratic cost function any standard  solution $\pi$ is concentrated on the graph of a mapping $T$. In this case $T$ is a solution of the corresponding Monge problem which asks for a mapping minimizing the  functional $\int c(x,T(x)) d \mu$ in the class of measure preserving
(i.e. pushing forward $\mu$ onto $\nu$) mappings.

Since its revival at the end of eighties the  Monge--Kantorovich theory attracts growing attention.
The reader can find a lot of information on the classical mass transportation theory   in many recent textbooks and survey papers
\cite{BoKo}, \cite{Galichon}, \cite{McCannGuill},  \cite{RR}, \cite{Villani}, \cite{Villani2}.

Our research is motivated by a number of recent results appeared in  several  quickly developping branches
of the mass transportation theory. Here is a short outline of the most important problems and ideas.

1)  Multimarginal transportation problem.

The book of Rachev and R{\"u}schendorf \cite{RR} contains rich material on the multimarginal transportation problem,
in particular, a number of functional-analytical results on duality, probabilistic applications etc.
However, until recently only the  two-marginals case was important for the largest part of applications. 
The books of Villani \cite{Villani}, \cite{Villani2} deal with the most important but specific  two-marginals case.

The revival of interest to the multimarginal Monge--Kantorovich problem  is partially motivated by economical applications
(matching theory, multi-dimensional screening), see \cite{McCannGuill}, \cite{Galichon}.  We refer to  survey paper 
\cite{Pass}.  Many references on recent works on multimarginal duality theory for a wide class of cost functions can be found in \cite{BoKo}.

2) Doubly- and multistochastic measures.

According to the classical  Birkhoff-von Neumann theorem every bistochastic matrix is a convex combination of 
permutation matrices. More precisely, the permutation matrices are exactly the extreme points  
of the set of bistochastic matrices. The classical problem of Birkhoff asks for a generalization of this result 
for the set of bistochastic (doubly stochastic) measures $\Pi(\mu,\nu)$. This problem has been attacked by many researches 
(see \cite{S}, \cite{BS}, \cite{HW}, \cite{AKMC}),
let us in particularly mention the seminal paper \cite{HW}, containing a characterization of supports of such measures. 
Using this characterization Ahmad, Kim, and McCann obtained in \cite{AKMC} interesting results on uniqueness of solution to the optimal transportation problem.
Exposition of relations between bistochastic measures, Markov operators, and Markov chains can be found in \cite{Vershik}.

In this paper we deal with $(n,k)$-stochastic measures, which are probability measures on a product space $$X = X_1 \times X_2 \times \ldots \times X_n$$ with fixed projections $\mu_{I}  \in \mathbb{P}(X_{I})$ 
for every $X_{I} = X_{i_1} \times \ldots \times X_{i_k}$, where $I = \{i_1, \ldots, i_k\}$ is a  $k$-tuple of indices, $k<n$.
The simplest (and most famous) example of such measures is given by the set of latin squares which is homeomorphic
to $(3,2)$-stochastic matrices.
It is important to emphasize that for the set of $(n,k)$-stochastic matrices (measures) with $n>2$
there is no analog of  the Birkhoff-von Neumann  theorem  (see \cite{Kil}, \cite{LL}, see also  \cite{CLN} for description of extreme points
for $k=2, n=3$ in the discrete case).

3) Monge--Kantorovich problem with linear constraints and  Monge--Kantorovich problem with symmetries.

Apparenty the most famous example of a transportation problem with linear constraints 
is the optimal martingale transportation problem coming from financial mathematics.  This problem is obtained from the classical one by adding an additional 
constraint: the measure $\pi$ is assumed to be a martingale (to make the space of feasible measures non-empty $\mu$ should stochastically dominate $\nu$). 
The dual martingale problem has a natural financial interpretation (see \cite{BHLP}).
More information about martingale transportation  the reader can find in \cite{BeigJuill},  \cite{BHLP}, \cite{GHLT}, \cite{HL}.
Remarkably, a  duality theorem for transportation problem with general linear constraints has been obtained only recently 
in \cite{Zaev}. This results covers, in particular, the case of martingale constraints. Another important class of linear constraints
are various symmetric assumptions, in particular, invariance with respect to an action of some group of linear operators.
This type of problem has been studied in \cite{Zaev}, \cite{Moameni}, \cite{GhM}.
Applications of symmetric problem to infinite-dimensional analysis and links with ergodic theory can be found in \cite{KZ1}, \cite{KZ2}.
The  Monge--Kantorovich problems with some convex constraints has been considered in  
\cite{Lev},  \cite{KorMC}, \cite{KorMCS}.

In this paper we consider the problem of maximization/minimization of the functional  
$$
\int_{X} c(x_1, \ldots,x_n) d \pi
$$
 on a set of $(n,k)$-stochastic measures. We call it mutistochastic or $(n,k)$-stochastic Kantorovich problem.
 Clearly, for $k=1$ one gets the multistochastic Kantorovich problem with $n$-marginals.
 
 The system of projections $\{\mu_I\}$ can not be arbitrary for $k>1$, and in fact, it is a nontrivial question, 
 when the set  of $(n,k)$-stochastic measures is non-empty. 
 This problem illustrates the main source of difficulties for the multistochastic problem: the constraints are highly
 non-independent, unlike the classical Monge--Kantorovich problem. 
 We stress that existence of feasible measures is just one question among  many others which have trivial solutuions for the classical case, but not for the multistochastic one.
  On the other hand, the classical example of latin squares and its relation
 with discrete algebraic structures (groups and quasigroups) ensures that it is an interesting and non-artificial object.
 
 We start with consideration of two basic questions of the mass transportation theory: duality and cyclical monotonicity.
 A natural guess that the dual problem should be the maximization problem
 $$
 \sum_{I} \int f_{I}(x_{i_1}, \ldots, x_{i_k}) d \mu_{I}, 
 $$
 with the constraint $\sum_{I}  f_{I}(x_{i_1}, \ldots, x_{i_k}) \le c$ is verified in Section 3 in a form analogous to the duality theorem considered
 in \cite{Villani}. The proof is based on the minimax principle. In Section 4 we prove an analog of the cyclical monotonicity property.
Unfortunately, applications of the cyclical nonotonicity are not that as straightforward as in the classical case.
 The main difficulty here is that the set of discrete competitors 
 is essentially more  complicated that the  permutation cycles considered in the classical transportation theory. We don't know, whether any solution to a multistochastic problem 
 (for a reasonable choice of the cost function $c$) is concentrated on the graph of a mapping (this is a standard corollary of the cyclical monotonicity property in the classical case). The uniqueness question is open as well. We were able, however, to deduce
 from the cyclical monotonicity property that any solution is singular to the Lebesgue measure 
 under assumption that the projections have densities ($k=2$, $n=3$, $X_i = [0,1]$).
 
 In Sections 5-6 we study our main example: $k=2$, $n=3$, $X_i = [0,1]$,
 the two-dimensional projections are assumed to be Lebesque and 
 $$
 c(x,y,z) = xyz.
 $$
 Let us consider the maximization problem $\int_{[0,1]^3} xyz \to \max$.
 We show that there exists a solution which is concentrated on the graph of the mapping 
 $$
 (x,y) \to 1 - x \oplus y,
 $$
 where $\oplus$ is the bitwise addition, which is also called xor-addition or Nim-addition.
 
 Similarly, for the minimization problem $\int_{[0,1]^3} xyz \to \min$
the solution $\pi$  is concentrated on the graph of the mapping 
$$
(x,y) \mapsto x \oplus y.
$$ 

It is known that the bitwise operations can be used to generate fractals 
(see \cite{Gib}, \cite{FK}). In particular, the 
 graph of this mapping 
$(x,y) \mapsto x \oplus y$ is 
 the so-called Sierpi\'nski tetrahedron.

\centerline{\includegraphics[scale = 0.4]{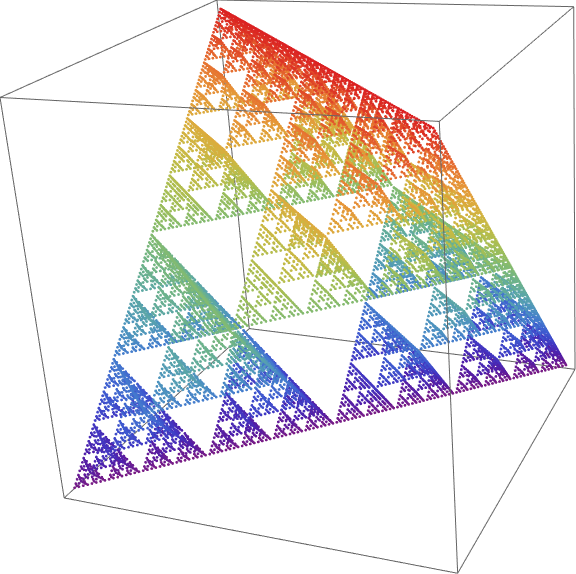}}

This is a classical   fractal self-similar set of dimension 2.
In the book of Mandelbrot \cite{M} it is briefly  described under the name ``fractal skewed web'' :
``Let us project it
along a direction joining the midpoints of either
couple of opposite sides. The initiator
tetrahedron projects on a square, to be called
initial. Each second-generation tetrahedron
projects on a subsquare, namely $1/4$-th
initial square, etc. Thus, the web projects on
the initial square. The subsquares' boundaries
overlap.''. 

  The irregularity of this example is rather unexpected, since it is well-understood that the standard  solutions to 
 the classical Monge--Kantorovich problems are supported by regular surfaces. On the other hand
the close relation of latin squares to groups makes the appearence of the xor-operation (equivalently, of the group $\mathbb{Z}_2^{\infty}$) natural.
The appearence of the bitwise addition can be also  illustrated by the baby $(3,2)$-transportation problem on the  cube
$\{0,1\}^3$ with $c=xyz$. All the competitors with uniform projections are convex  combinations of two measures: $$\{(0,0,0), (0,1,1), (1,1,0), (1, 0,1) \},$$ $$\{(0, 1, 0), (0, 0,1), (1,0,1), (1,1,1)\}$$ (we identify the point and the 
Dirac mass with weight $1/4$ at the point), defined by equations
$$
 x + y + z=0, \  x +y + z =1, \ mod(2).
$$
The first measure minimizes $xyz$ and the second measure maximizes.
The fact that essentially the  same structure is preserved for the cube $[0,1]^3$ due to the symmetries of the corresponding continuous problem.

Let us consider again the minimization problem $\int_{[0,1]^3} xyz \to \min$. We show in Section 7 
that
$$
F(x,y) = \int_{0}^x \int_{0}^y s \oplus t \ ds dt - \frac{1}{4} \int_{0}^x \int_{0}^x s \oplus t \ ds dt
- \frac{1}{4} \int_{0}^y \int_{0}^y s \oplus t \ ds dt.
$$
solves the corresponding dual problem
$$
\int_{[0,1]^2} F(x,y) dx dy  + \int_{[0,1]^2} F(x,z) dx dz + \int_{[0,1]^2} F(y,z) dy dz  \to \max,
$$
$$F(x,y) + F(x,z) + F(y,z) \le xyz.$$
In particular, the corresponding optimal mapping $T$ takes the form
$$
(x,y) \mapsto \partial^2_{xy} F(x,y)
$$
and the Sierpi\'nski tetrahedron is the set of zeroes  of the nonnegative function 
$$
xyz - F(x,y) - F(x,z) - F(y,z).
$$
In addition, this function is almost everywhere differentiable and homogeneous with respect to factor $2$.
The first derivatives of this function are not differentiable, but have bounded variation.

It is an open question which particular properties of this solution are inherited by general solutions to the (3,2)-problem.
We  discuss some related hypotheses in Section 8.

\section{Multistochastic problem. Basic properties.}

In this short section we define the main objects of our study and discuss their basic properties.
We are given a finite number of spaces
$$
X_1 , \ldots, X_n,
$$
equipped with 
$\sigma$-algebras
$$
\mathcal{B}_1, \ldots, \mathcal{B}_n.
$$
The product space
$$
X = X_1 \times \ldots \times X_n,
$$
is equipped  with the standard  product of $\sigma$-algebras
$$
\mathcal{B} = \mathcal{B}_1 \times \ldots \times \mathcal{B}_n.
$$
The projection 
$$
X \ni x \to (x_{i_1}, \ldots, x_{i_k})
$$
of $X$ onto 
$
X_{i_1} \times \ldots X_{i_k}, \ i_j \in \{ 1, \ldots, n\}$ with $i_{j_1} \ne i_{j_2}$ for distinct $j_1, j_2$,
will be denoted by
$$
{\rm Pr}_{X_{i_1} \times \ldots \times X_{i_k}}, \ {\rm Pr}_i,
$$
where $i = (i_1, \ldots, i_k)$.

Thoroughout the paper the following assumption holds:

{\bf Assumption I.}  $X_i$ are Polish spaces and $\mathcal{B}_i$ are the corresponding Borel $\sigma$-algebras.

\begin{definition}
\label{msK}
{\bf (Multistochastic Kantorovich problem)}

For every fixed $1 \le k < n$ let  $\mathcal{I}_k$ 
be the set of all ordered
$k$-tuples of indices $i_j \in \{1, \ldots, n\}$, $i_1 < i_2 <\ldots < i_{k-1} < i_k$. Assume that for every $k$-tuple
$$
I = (i_1, \ldots, i_k) \in \mathcal{I}_k
$$ we are given 
a probability measure $\mu_{I} = \mu_{i_1, \ldots, i_k}$ on $X_{i_1} \times \ldots \times X_{i_k}$.
Denote by $\mathcal{P}_{\mu}$ the set of probability measures on $X$ satisfying
$$
{\rm Pr}_I \mu := {\rm Pr}_{X_{i_1} \times \ldots \times X_{i_k}} \mu = \mu_{I}.
$$
Finally, assume that we are given a cost function  $$c \colon  \prod_{i=1}^n X_i \to \mathbb{R}_{+} \cup \{+\infty\}.$$

Then we say that $P \in \mathcal{P}_{\mu}$ is a solution to the $(n,k)$-Kantorovich minimization problem for $c$ and
$\{\mu_i\}, I \in \mathcal{I}_k$, if $P$ gives minimum to the functional
$$
P \to \int_{X} c  \ d P
$$ 
on $\mathcal{P}_{\mu}$. 

We call the problem ``$(n,k)$-Kantorovich maximization problem''
if instead of minimum we are looking for maximum of
$
P \to \int_{X} c  \ d P.
$  
\end{definition}

Unlike the standard   Kantorovich problem  $\mathcal{P}_{\mu}$ can be empty. Let us briefly discuss some sufficient 
conditions assuring that  $\mathcal{P}_{\mu}$ is not empty. For the sake of simplicity we restrict ourselves
to the case $n=3, k =2$, $X_i = [0,1]$. 

A natural necessary assumption for $\mathcal{P}_{\mu} \ne \emptyset$ is the following Kolmogorov-type consistency condition.

\begin{remark}
If the set $\mathcal{P}_{\mu}$ is not empty, then 
\begin{equation}
\label{consist}
{\rm Pr}_{1} \mu_{1,2} = 
{\rm Pr}_{1} \mu_{1,3}, \ 
{\rm Pr}_{2} \mu_{1,2} = 
{\rm Pr}_{2} \mu_{2,3}, \
{\rm Pr}_{3} \mu_{2,3} = 
{\rm Pr}_{3} \mu_{1,3}.
\end{equation}
\end{remark}

\begin{remark}
One can naively think that (\ref{consist}) is a sufficient condition for  $\mathcal{P}_{\mu} \ne \emptyset$.
But this is not true. Consider the following example: 
$\mu_{1,2}$  is given by the (normalized)  Lebesgue measure on the diagonal $\{x=y\}, x \in [0,1], y \in [0,~1]$ and $\mu_{1,3}$
is the two-dimensional Lebesgue measure on $[0,~1]^2$. 

Since the projection of the set $\{x=y\}\times [0,~1]$ onto $0 \le x \le 1,~0 \le z \le 1$ along $y$ is a one-to-one mapping, there exist the unique measure on $[0,~1]^3$ with projections $\mu_{1,2}$, $\mu_{1,3}$. Denote this measure by $\pi$. 

In this construction $\mu_{2,3}$ was not used, so if $\mu_{2,3} \ne Pr_{2,3}(\pi)$, there exist no measure on $I$ with projections $\mu_{1,2}$, $\mu_{1,3}$ and $\mu_{2,3}$. But one can easily find $\mu_{2,3}$ different from $Pr_{2,3}(\pi)$ such that (2.1) holds.

\end{remark}

\begin{remark}
An important example of a non-empty set $\mathcal{P}_{\mu}$ is given by the following system of projections ( for the sake of simplicity $k=2$):
$$
\mu_{i,j} = \mu_i \times \mu_j, 
$$
where every $\mu_i$ is a probability measure  on $X_i$.
\end{remark}

{\bf Assumption II.}
It will be assumed throughout that  $\mathcal{P}$ is non-empty.

The proof of the following result is omitted because it is a simple repetition of the proof of the corresponding fact
for the standard Kantorovich problem (see \cite{BoKo}, \cite{Villani}).

\begin{theorem}
Assume that  $c$ is a lower semicontinuous function. Then there exists 
$P \in \mathcal{P}_{\mu}$ giving minimum to the functional $P \to \int c d P$ on  $\mathcal{P}_{\mu}$.
\end{theorem}

\section{Duality}

In this section we prove a duality theorem for the multistochastic problem.
It can be deduced
from the following general minimax result   (see \cite{Villani}, Theorem 1.9)
in the same way as the duality theorem for the standard  Kantorovich problem.
The arguments are essentially the same, we repeat the proof for the reader convenience.
For the sake of simplicity we restrict ourselves to the case of compact spaces.

\begin{theorem}
\label{abstract-dual}
Let $E$ be a normed vector space and $E^*$ be the corresponding topologically  dual space. Consider convex functionals
$\Phi, \Psi$ on  $E$  with values in $\mathbb{R} \cup \{+ \infty\}$.
Let $\Phi^*, \Psi^*$ be their Legendre transforms. 
Assume that there exists a point  $z \in E$ satisfying
 $\Phi(z) < + \infty, \Psi(z) < + \infty$ and $\Phi$ is continuous at $z$.
Then
$$
\inf_{E} (\Phi + \Psi) = \max_{z \in E^*}  (-\Phi^*(-z) - \Psi^*(z) )
$$
\end{theorem}

\begin{theorem}
\label{dualth}
Let $X_i$ be compact metric spaces and  $c \ge 0$ be a continuous function on $X$.
Then
$$
\min_{\pi \in \mathcal{P} } \int c d \pi = \sup\sum_{i \in \mathcal{I}_k} \int  f_{i} (x_{i_1}, \ldots, x_{i_k}) d \mu_{i}, 
 $$ 
 there the infimum is taken over the $k$-tuples  $i=(i_1, \ldots, i_k) \in \mathcal{I}_k$ and the functions $f_{i_1, \ldots, i_k} \in
 L_1(X_{i_1} \times \ldots \times X_{i_k}, \mu_{(i_1, \ldots, i_k)})$
satisfying
 $$
 \sum_{i \in \mathcal{I}_k} f_{i} (x_{i_1}, \ldots, x_{i_k}) \le c.
 $$
\end{theorem}

\begin{proof}
Let $E$ be the space of continuous functions on $X$.
 By Radon's theorem $E^*$ is the space of finite (signed) measures on $X$.

Set:
$$
\Phi(u) = 0,\ \mbox{\rm{if}} ~ \ u \ge -c
$$
and $\Phi(u)=+\infty$ in the opposite case.

Let  $\pi_0$ be a probability measure which belongs to $\mathcal{P}_{\mu}$. For every function $u$ which has representation $u = \sum_{i \in \mathcal{I}_k} f_{i}$
we set
$$
\Psi(u) = \int u d \pi_0 =  \sum_{i \in \mathcal{I}_k}\int  f_{i} d \mu_{i}, \ \ \mbox{if} ~ u = \sum_{i \in \mathcal{I}_k} f_{i},
$$
and $\Psi(u)=+\infty$ in the opposite case.
It is easy to check that the functionals satisfy assumptions of Theorem
\ref{abstract-dual}.

Clearly
$$
\inf_{u} \Bigl( \Phi(u) + \Psi(u) \Bigr) = -\sup_{ \sum_{i \in \mathcal{I}_k} f_{i} \le c}  \sum_{i \in \mathcal{I}_k}\int  f_{i} d \mu_{i}.
$$
Let us find the Legendre transform of the functionals
$$
\Phi^*(-\pi)  = \sup_{u} \Bigl( -\int u d \pi - \Phi(u) \Bigr) =  \sup_{u \ge -c} \Bigl( -\int u d \pi  \Bigr) = \sup_{u\le c} \int u d \pi.
 $$
It is easy to see that $\Phi^*(-\pi)  = \int c d \pi$, if $\pi$ is a non-negative measure. If not, then clearly $\Phi^*(-\pi)  = +\infty$.
 
Let us compute
 $
 \Psi^{*}(\pi)  =  \sup_{u} \Bigl( \int u d \pi - \Psi(u) \Bigr).
 $
Clearly $\Psi^{*}(\pi) = 0$, if ${\rm Pr}_{i}  \pi= \mu_{i}$ and $\Psi^{*}(\pi) = +\infty$
in the opposite case. This implies
$$
 \max_{z \in E^*}  (-\Phi^*(-\pi) - \Psi^*(\pi) )= -\min_{Pr_{i}  \pi= \mu_{i}  } \int c d \pi.
$$
The proof is complete.
\end{proof}

\section{Cyclical monotonicity}

Starting from this section we work with the following particular case:
$$
n=3, k=2,
$$
$$
c(x,y,z) = xyz.
$$ 
Here we consider the {\bf maximization} problem
$$
\int_{[0,1]^3} xyz d  \pi \to \max, \ \pi \in \Pi(\mu_{12}, \mu_{23}, \mu_{13}).
$$

This problem seems to be the simplest but important and illustrative particular case of the 
multistochastic Kantorovich problem.
The choice of the cost function is natural in view of the examples which will be given below.
In addition, it is analogous to the simplest
quadratic Kantorovich problem with one-dimensional marginals.
Indeed, the minimization of 
$\int |x-y|^2 d \pi$ on the set of measures $\Pi(\mu,\nu)$ with fixed marginals $\mu,\nu$
is equivalent to maximization of $\int xy d \pi$ on the same set.

 In the case of the standard Kantorovich problem with two marginals
 the well-known cyclical monotonicity property fully characterizes the solutions.
 In particular, if the marginals are one-dimensional, then the solution is concentrated on the graph 
 of a monotone function.
In this section we prove a weak analog of this property for our special multistochastic problem. Unlike the standard Kantorovich problem with one-dimensional marginals,
the geometric structure of the sets which are cyclically monotone in our sense is essentially less clear.

 To show cyclical monotonicity we follow approach from \cite{BeigJuill} (Lemma 1.11).

Assume we are given three finite sets 
$$
X = \{x_1, \ldots x_n\} \subset \mathbb{R}, \ 
Y = \{y_1, \ldots y_n\} \subset \mathbb{R}, \
Z = \{z_1, \ldots z_n\} \subset \mathbb{R}
$$
of cardinality $n$

In what follows we denote by 
$$U(X,Y,Z)
$$ the set of discrete probability measures on $X \times Y \times Z$ which have uniform projections onto 
$X \times Y, X \times Z, Y \times Z$.
Among of  measures from $U(X,Y,Z)$ let us consider a special important subclass of uniform distributions $\pi_{G}$ on the sets of the type
$$
G = (x,y, f(x,y)) , \  x \in X, \ y \in Y, \ f \colon X\times Y \to Z,
$$
where $f$ admits the following property:
fix any $z_i \in Z$, then for every  $x_j \in X$ there exists exactly one $y_k \in Y$ and 
 for every  $y_l \in Y$ there exists exactly one $x_m \in X$ such that
$$
z_i = f(x_i,y_k) = f(x_m,y_l).
$$
Then the uniform measure $\pi_{G}$ on $G$ belongs to $U(X,Y,Z)$.

The set 
$$
L(X,Y,Z)
$$ 
of such $G$ can be identified with $n \times n$ {\bf latin squares}.

\begin{remark}
Let us mention another important difference between the multistochastic and the standard Kantorovich problem.
By the classical theorem of Birkhoff every bistochastic matrix is a convex combination of the  permutation matrices.
In the multistochastic case there is no analog of the Birkhoff theorem:
not every multistochastic matrix is a convex combination of matrices with entries $(a_{ij})$ satisfying $a_{ij}=0$ or $a_{ij}=1.$
An example is given by 
$$
\includegraphics[scale = 0.7]{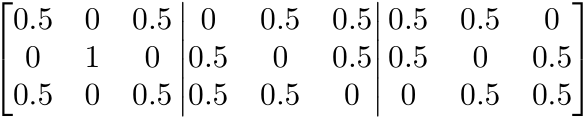}
$$						
See \cite{LL} for explanations and \cite{CLN} for  descriptions of extremal points.

\end{remark}

\begin{definition}
Let $\Gamma \subset \mathbb{R}^3$ be a set and $\pi_{\Gamma}$ be the uniform measure on $\Gamma$.
We say that $\Gamma'$ is the competitor of $\Gamma$ if $\pi_{\Gamma}$ and $\pi_{\Gamma'}$
have the same projections onto the principal hyperplanes
$$
Pr_{xy} \Gamma = Pr_{xy} \Gamma', \ Pr_{yz} \Gamma = Pr_{yz} \Gamma', \ Pr_{xz} \Gamma = Pr_{xz} \Gamma'. 
$$
\end{definition}

\begin{definition}
The set $\Gamma \subset \mathbb{R}^3$ is called cyclical monotone if for every natural $n$ and  $G \subset \Gamma$ of cardinality $n$
$$
\int xyz d \pi_{G'} \le  \int xyz d \pi_{G}
$$
for every competitor of $G$.
\end{definition}

\begin{theorem}
\label{cm}
Let $\pi$ be a solution to a $(3,2)$-multistochastic Kantorovich maximization problem with $c(x,y,z) = xyz$.
Then there exists a cyclical monotone set $\Gamma$ with $\pi(\Gamma)=1$.
\end{theorem}
\begin{proof}
For every $n$ let 
\begin{align*}
M_{n}  = & \Bigl\{ G  \subset \mathbb{R}^3, {\rm {card}}(G)=n,\  \text{there exists a competitor $G'$ such that} \
\\& \int xyz d \pi_G < \int xyz d \pi_{G'}  \Bigr\} 
\subset (\mathbb{R}^3)^{n}.
\end{align*}
According to a well-known result of H.~Kellerer (see \cite{Kell}) one of the following two options holds:

(1) $M_n$ is contained in a set of the type $\cup_{i=1}^{n} \mathbb{R}^3 \times \ldots \times M^{i}_n \times \ldots \times \mathbb{R}^3$ with $\pi(M^i_n)=0$.

(2) There exists a measure $\gamma$ on $M_n$ such that $\gamma(M_n)>0$  and ${\rm Pr}_i (\gamma) \le \pi$ for every $i$.

We will show that (2) is impossible. Thus (1) holds for every $n$ and $$\Gamma = \cap_{i=1}^{n} \mathbb{R}^2 \setminus \cup_{i=1}^{n} M^i_n$$ is the desired set.

Assume that (2) holds for some $n$. Without loss of generality let us assume that $\gamma \le \frac{1}{n} \pi$.
Set $\gamma' = \sum_{i=1}^{n} {\rm Pr}_i \gamma$. Clearly
$$
\gamma'  = \int \pi_{G} d \gamma(G).
$$
By definition of $M_n$ for $\gamma$-a.e. $G$ there exists a competitor  $G'$  such that
and $\int xyz d \pi_{G} < \int xyz d \pi_{G'}$. Moreover, using linear programming algorithms one can make the correspondence
$G \to G'$ measurable.
Define 
$$
\tilde{\gamma} = \int \pi_{G'} d \gamma(G)
$$
Clearly $\int xyz d \tilde{\gamma} > \int xyz d {\gamma'}$ and $\gamma', \tilde{\gamma}$ have the same projection onto the principal
hyperplanes. Then we set $\pi' = \pi - \gamma' + \tilde{\gamma}$. The measures $\pi, \pi'$ have the same projections  onto the principal
hyperplanes and $\pi'$ has a larger total cost. We obtain a contradiction.
\end{proof}

\begin{example}
\label{4p}
The simplest example of a set $G$ which belongs to some $L(X,Y,Z)$ is given by the following four-points set
with uniform projections on the products of some two-points sets
$$
 X_1 = (a_1, b_1, c_2),  X_2 = (a_1, b_2, c_1), X_3 =(a_2, b_1, c_1), X_4 =(a_2, b_2, c_2).
 $$
 The set $G$ is cyclically monotone for $c=xyz$
 if and only if 
 $$
  (a_1-a_2) (b_1 - b_2)(c_1-c_2) \le 0.
  $$
\end{example}

The well-known and by now classical result of Y.~Brenier establishes 
existence of the so-called optimal transportation mapping in the classical setting.
We don't know whether the multistochastic Kantorovich problem admits the same property.
However, applying the cyclical monotonicity property proved in Proposition \ref{cm}
we are able to show a weak version of the Brenier theorem saying that under natural assumptions
$\pi$ is a singular measure.

Let us denote by $\lambda^n$ the standard $n$-dimensional Lebesgue measure.

\begin{lemma}
\label{cube}
Let  $A \subset \mathbb{R}^3$  be a Borel set of positive Lebesgue measure.
There exist numbers 
$$
x_1 < x_2, y_1 < y_2, z_1 < z_2,
$$
such that  $\{x_1, x_2\} \times \{y_1, y_2\} \times \{z_1, z_2\} \subset A$.
\end{lemma}
\begin{proof}
Without loss of generality let us consider bounded sets.
By Fubini's theorem one gets that for every $\varepsilon>0$ 
the set  $A_z$ of numbers $\tilde{z}$ satisfying
$$
\lambda^2(A \cap \{z = \tilde{z} \} ) > \varepsilon
 $$ 
 has a non-zero Lebesgue measure. Hence there exists two points 
 $z_1, z_2 \in A_z$, such that the projections $A \cap \{z = z_1\}$,  $A \cap \{z = z_2\}$
 onto the hyperplane $xy$ have an intersection $B$ of a positive measure. Hence
 $B \times \{z_1,z_2\} \subset A$. Next we apply the same arguments to the one-dimensional sections of $B$:
 $\{y=\tilde{y} \} \cap B$. This completes the proof.
\end{proof}

\begin{corollary}
\label{singular}
Every cyclical monotone set $\Gamma \subset \mathbb{R}^3$ satisfies  $\lambda^3(\Gamma)=0$.
\end{corollary}
\begin{proof}
Assume that $\lambda^3(\Gamma)>0$. Then according to Lemma \ref{cube} there exist numbers $x_1 < x_2, y_1 < y_2, z_1 < z_2,
$
such that  $\{x_1, x_2\} \times \{y_1, y_2\} \times \{z_1, z_2\} \subset \Gamma$. We get a contradiction with
Example \ref{4p}.
\end{proof}

  \section{Main example. Primal problem.}

 In this  section we consider our main example: $(3,2)$-Kantorovich problem
 on the unit three-dimensional cube  $[0,1]^3$,  where the projections onto principal hyperplanes
 are equal to two-dimensional Lebesgue measure $\lambda^2$. The  cost function is given by
 $$
 c(x,y,z) = xyz.
 $$
 The set of measures with such projections will be denoted by $\mathcal{P}_{\lambda}$.
We are looking for 
 \begin{equation}
 \label{3-2-leb}
\max_{P \in \mathcal{P}_{\lambda}} \int xyz \ dP.
 \end{equation}
 In this concerete example we are able to find an explicit solution. We emphasize that this is possible because the 
 problem admits many symmetries. We don't know whether the problem has an explicit solution even after
 slight changes, for instance, in the case when the projections are equal to products $\mu_i \times \mu_j$,
 where $\{\mu_i\}$ are fixed one-dimensional distributions.

 We denote by $\oplus$ the bitwise addition (xor). Given two couples of numbers
 $x, y \in [0,1]$ we consider their diadic decompositions 
 $$
 x =\overline{0, x_1 x_2 x_3 \ldots}, \  y = \overline{0, y_1 y_2 y_3 \ldots}, \ \ x_i, y_i \in \{0,1\}.
 $$
 Then the xor operation is defined as follows:
 $$
 x \oplus y =  \overline{ 0, x_1 \oplus y_1 \  x_2 \oplus y_2  \ x_3 \oplus y_3 \ldots},
 $$
 where $0 \oplus 0 = 1 \oplus 1 =0, 0 \oplus 1 = 1 \oplus 0=1$.
 
 \begin{remark}
 The addition is not well-defined for dyadic rational numbers, because they can be written in two different ways.
 We agree that every dyadic rational number less then $1$  has a finite numbers of units in its decomposition. The number $x=1$ will be always decomposed in the following way:
 $$
 1 =  \overline{ 0, 11111 \ldots}.
 $$
 Thus
 $$
 x \xor 1 = 1 -x.
 $$
 This operation is continuous up to a countable set of dyadic numbers. 
 \end{remark}

 \begin{theorem}
 The image $\pi$ of  two-dimensional Lebesgue  measure $\lambda(dx) \times \lambda(dy)$
 under the mapping $$T : (x,y) 
 \to (x,y, 1-x\oplus y)$$
 is a solution to problem (\ref{3-2-leb}).
 
 If instead of maximizing the total cost function one asks for 
 $$
\min_{P \in \mathcal{P}_{\lambda}} \int xyz \ dP,
$$
 then the corresponding mapping $T$ is given by
 $$T : (x,y) 
 \to (x,y, x\oplus y).$$
 \end{theorem}
 
 \begin{remark}
 We don't know whether  this concrete problem  and the problem in general setting (for an appropriate cost function) has unique solution.
 In this example there exists a corresponding optimal mapping, 
 but we don't know whether the same is true for any $(3,2)$-problem (under appropriate assumptions 
 on the projections).
 \end{remark}
 
 \begin{proof}
 Let us consider the following transformations of $[0,1]^3$
 $$
 T_{xy}(x,y,z)  = (1-x,1-y,z),
 $$
 $$  
 T_{xz}(x,y,z)  = (1-x,y,1-z),  
 $$
 $$
  T_{yz}(x,y,z)  = (x,1-y,1-z).
 $$
 All these transformations push forward arbitrary measure $\mu \in \mathcal{P}_{\lambda}$ onto a measure from $\mathcal{P}_{\lambda}$. We define
 $$
 \mu^{xy} = \mu \circ T^{-1}_{xy}, \ \mu^{xz} = \mu \circ T^{-1}_{xz}, \ \mu^{yz} = \mu \circ T^{-1}_{yz}.
 $$
 
 Next we note that every $\mu \in \mathcal{P}_{\lambda}$ satisfies
 \begin{align*}
 \int xyz d \mu^{xy} & = \int (z - xz - yz + xyz) d \mu \\& = \int xyz d \mu + \int_{0}^1 z dz 
 - \int_{0}^1 \int_{0}^{1} xz dx dz  - \int_{0}^1 \int_{0}^{1} yz dy dz = \int xyz d \mu .
 \end{align*}
Thus the total cost $\int xyz d \mu$ is invariant with respect to $T^{xy}$
 (and with respect to $T^{yz}, T^{xz}$).
 Hence it follows that for every $\tilde{\pi}$  solving (\ref{3-2-leb}) the measures $\tilde{\pi}^{xy}, \tilde{\pi}^{yz}, \tilde{\pi}^{xz}$,  and
 $$
{\pi}_1= \frac{\tilde{\pi} + \tilde{\pi}^{xy} + \tilde{\pi}^{xz} + \tilde{\pi}^{yz}}{4}
 $$
 are solutions to problem (\ref{3-2-leb}) as well. Note that   ${\pi}_1$
 is invariant with respect to   $T^{xy}$, $T^{yz}$, $ T^{xz}$. This follows from the relations
$$
T^{xy} T^{xz}  = T^{xz} T^{xy}   = T^{yz} , \  T^{xy} T^{xy} = \rm{Id}.
$$
 
Next we decompose  $[0,1]^3$ into sets   $I_1$, $I_2$. Every $I_i$, $i \in \{1,2\}$
 is a union of four smaller cubes of volume $1/2^3$:
$$
I_1 = \overline{[0,1]^3 \setminus I_2}
$$

$$
I_2 = \Bigl[0, \frac{1}{2}\Bigr]^3 \bigcup \Bigl( \Bigl[\frac{1}{2},1 \Bigr]^2 \times \Bigl[0, \frac{1}{2}\Bigr] \Bigr) \bigcup
\Bigl( \Bigl[0,\frac{1}{2}\Bigr] \times \Bigl[ \frac{1}{2},1\Bigr]^2 \Bigr)
\bigcup \Bigl(  \Bigl[\frac{1}{2},1 \Bigr] \times  \Bigl[0, \frac{1}{2} \Bigr] \times  \Bigl[ \frac{1}{2},1\Bigr] \Bigr).
$$
 
Since every set $I_1, I_2$  is invariant under  $T_{xy}$, $T_{yz}$, $ T_{xz}$, the measures  
$$\pi_{I_1} = (\pi_1)|_{I_1}, \  \pi_{I_2} = (\pi_1)|_{I_2}$$
are invariant as well.
Hence the push-forward image 
 $$
 \pi^x_{I_2} =  \pi_{I_2} \circ T_x^{-1}
 $$
 of measure
$
 \pi_{I_2}
 $
 with respect to  $T_x \colon (x,y,z) \mapsto (1-x,y,z)$
 has the same hyperplane projections as $\pi_{I_2}$.
Thus
 $$
\pi_{I_1} +  \pi^x_{I_2}
 $$
 belongs to  $\mathcal{P}_{\lambda}$.
 
Let us show that $\pi_{I_2}=0$. To this end it is sufficient to show that 
 $$
 \int  xyz d \mu <  \int  xyz d \hat{\mu}, 
 $$
 where $\hat{\mu} = \mu \circ (T^x)^{-1}$,
 for every non-zero measure  $\mu$, which is invariant with respect to  $T^{xy}$, $T^{yz}$, $ T^{xz}$,
  and sastisfies $\rm{supp}(\mu) \subset I_2$. 
Indeed, if we show this, then we get
 $$
\int xyz d\pi_{I_1} + \int xyz d \pi^x_{I_2} > \int xyz d\pi_{I_1} + \int xyz d \pi_{I_2}.
 $$
The latter implies that measure 
$
\pi_{I_1} +  \pi^x_{I_2}
$ gives better value to the total cost function.
 
 Let $\nu$ be the projections of  $\mu$ (hence, projections of $\hat{\mu}$)
 onto $x$-axis, and $\eta^{x}(dy,dz)$,  $\hat{\eta}^{x}(dy,dz)$  are corresponding conditional measures
 $$
 \mu = \nu(dx) \eta^{x}(dydz), 
 $$
 $$
 \hat{\mu} = \nu(dx) \hat{\eta}^{x}(dy dz).
 $$
 Note that
 $\eta$ is invariant with respect to  $T^{yz}$  and
 \begin{align}
 \label{xy-eta}\hat{\eta}^x = \eta^x \circ T^{-1}_{y} =  \eta^x \circ T^{-1}_{z} = \eta^{1-x}  = \hat{\eta}^{1-x} \circ T^{-1}_{z}.
 \end{align}
 Hence
 \begin{align*}
 \int xyz (d\mu - d \hat{\mu}) & = \int \Big(\int yz (d \eta^x - d \hat{\eta}^x) \Big) x  \nu(dx)
 \\& =  \int_{0}^{\frac{1}{2}} \Big(\int yz (d \eta^x - d \hat{\eta}^x) \Big) x  \nu(dx)
 +  \int_{\frac{1}{2}}^1 \Big(\int yz (d \eta^x - d \hat{\eta}^x) \Big) x  \nu(dx)
 \\&
 = 
  \int_{\frac{1}{2}}^1 \Big(\int yz (d \eta^x - d \hat{\eta}^x) \Big) (2x-1)  \nu(dx).
 \end{align*}
 Next, using $T^{zy}$-invariance of  $\eta$ and (\ref{xy-eta}), one gets
 \begin{align*}
 \int yz & (d \eta^x - d \hat{\eta}^x) 
 = \frac{1}{2} \Bigl( \int (yz + (1-z)(1-y)) (d \eta^x - d \hat{\eta}^x) \Bigr)
 \\&
 =  \frac{1}{2} \int \big(yz + (1-z)(1-y) - (1-y)z - y(1-z)\big) d \eta ^x
 =  \frac{1}{2} \int (2y-1)(2z-1) d \eta^x. 
 \end{align*}
 Finally,
 
 \begin{align*}
 \int xyz (d\mu - d \hat{\mu}) 
 = 
 \frac{1}{2} \int_{\frac{1}{2}}^1 \bigr[\int (2y-1)(2z-1) d \eta^x(dzdy) \bigl] (2x-1)  \nu(dx).
 \end{align*}
Since the support of  $\mu$ lies in $I_2$, one gets 
$\int xyz (d\mu - d \hat{\mu})  < 0$.

Thus we get that the support of $\pi_1$  belongs to the union of four disjoint cubes with volume $1/2^3$
$$
J_1 = I_1 = C_1 \cup C_2 \cup C_3 \cup C_4.
$$
Hence the restriction of  $\pi_1$ onto every cube  $C_i$ is a solution of (\ref{msK})
for the same cost function with marginals which are restrictions of Lebesgue measure on projections of correspoding
$C_i$. Hence the same arguments are applicable  to every  $C_i$  and one gets 
a solution $\pi_2$ supported on a union of  16 cubes of volume $1/4^3$ 
$$
J_2 = \cup_{i=1}^4 \cup_{j=1}^4 C_{ij}.
$$
Reapeating this argument one gets a sequence of decreasing sets $J_n$ such that  each of them contains support of a measure
$\pi_n$ which solves (\ref{msK}). 
Clearly, the sequence $\{\pi_n\}$ admits a weak limit $\pi$ supported on
$$
J=\cap_{n=1}^{\infty} J_n.
$$
We get immediately that  $\pi$ solves the desired problem, moreover
 $J$ is a graph of  $T(x,y)$ (up to a set which projection on  $xy$ has zero measure) and $\pi$ is the unique measure
 supported on $J$ with the desired projections.
\end{proof}

The following pictures represent the iteration porocedure.

\includegraphics[scale = 0.3]{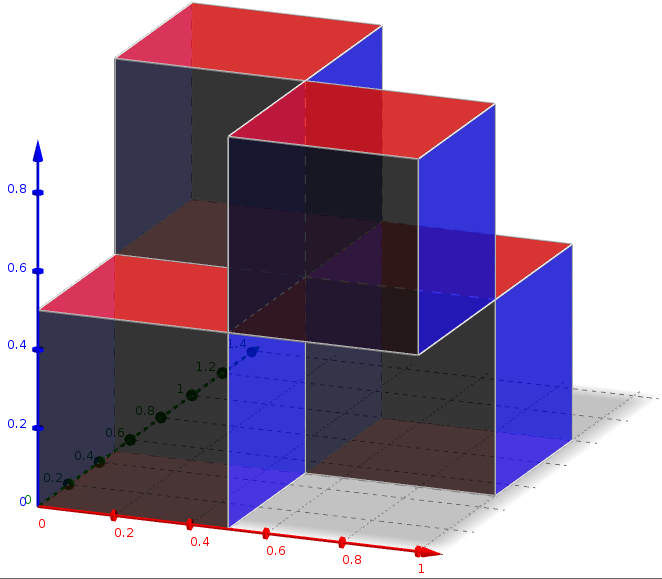}
\includegraphics[scale = 0.3]{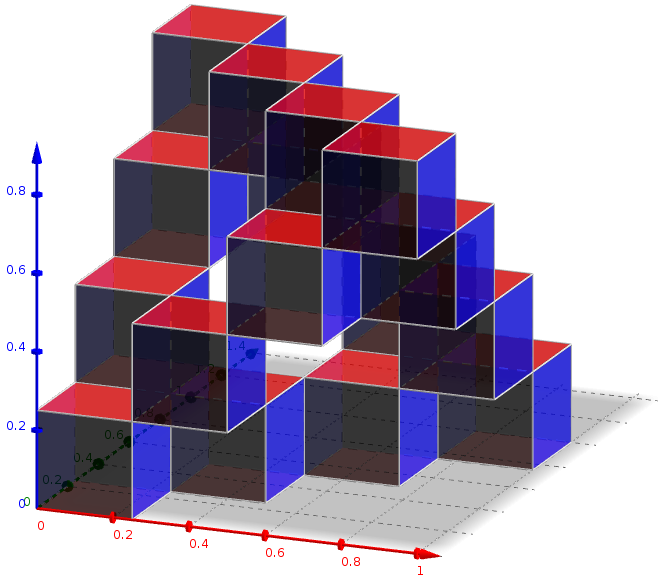}

\includegraphics[scale = 0.3]{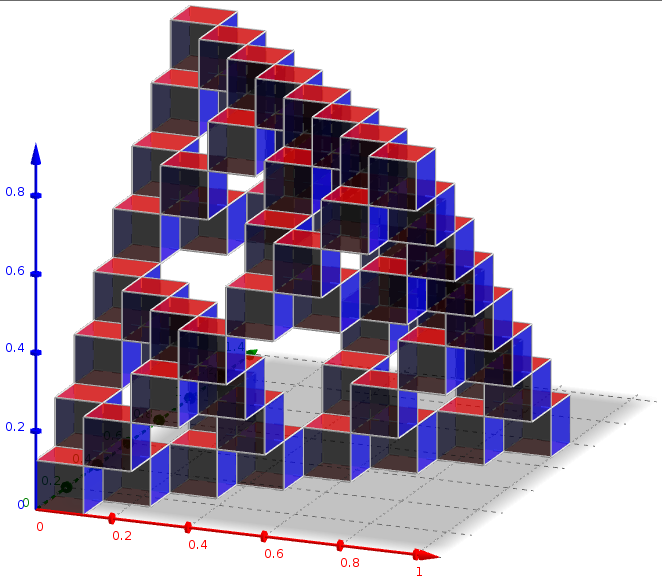}
\includegraphics[scale = 0.3]{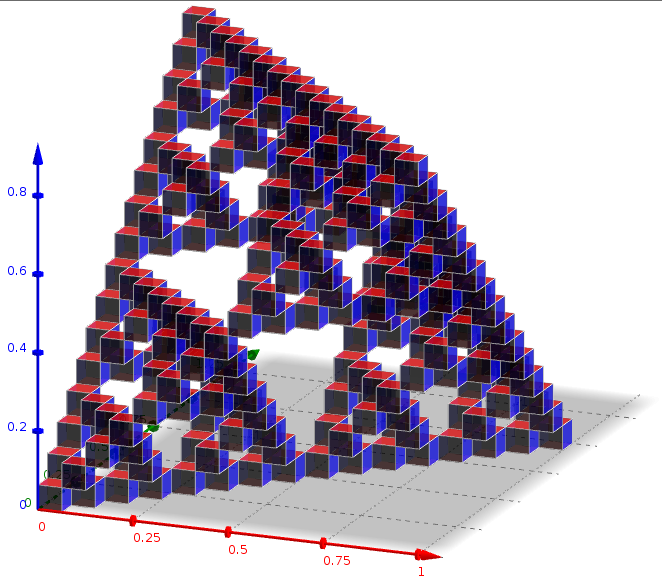}

As we already mentioned, $J$ is a  self-similar fractal of Hausdorff dimension two, called
``Sierpi\'nski tetrahedron''. This is a Kantor-type set which is a limit of iterations of unions of $4^n$ tetrahedrons.
Remarkably, in our proof we get an alternative construction and obtain $J$ as an intersection of collections of cubes.

\begin{remark}
The most trivial example of a fractal solution to the Monge-Kantorovich problem is apparently
 the $(3,2)$-Kantorovich problem with
Lebesgue measure  projections and $c = 1 - x \oplus y$. Then the solution is again the Sierpi\'nski tetrahedron.
But this due to a special choice of the cost function.
Unlike this, our main example deals with the smooth cost function $c=xyz$
and the extremality of the presented solution is highly non-obvious.
In addition, we will see in the subsequent sections that a solution to the corresponding dual problem
provides a non-trivial representation of the Sierpi\'nski tetrahedron as a set of zeroes of an a.e. differentiable function.

Less trivial example is given by measures supported on the set 
\begin{equation}
\label{ser-tr}
T = \{x + y  = x \xor y, \ x \in [0,1], \ y \in [0,1] \},
\end{equation}
which is a variant of the Sierpi\'nski triangle
(see \cite{FK}).

\centerline{ \includegraphics[scale = 0.2]{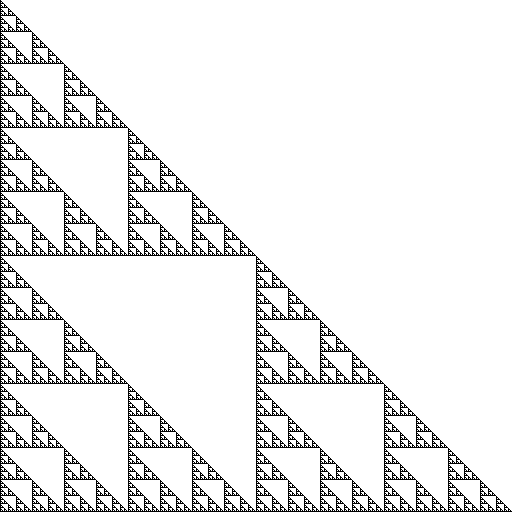}}

Note that all $x \in [0,1], \ y \in [0,1]$ satisfy
$$
x + y \ge x \xor y.
$$
Let $\pi$ be any probability measure on $T$ with projections $Pr_x \pi = \mu$, $Pr_y \pi = \nu$.
 Consider the Monge--Kantorovich problem
\begin{equation}
\label{x-y-xor}
\int_{[0,1]^2} x \oplus y \ d P \to \max,  \ Pr_x P = \mu, \ Pr_y \pi = P.
\end{equation}
By the Kantorovich duality principle the functions $x, y $ solve the corresponding dual problem.
Hence $\pi$ is a solution to (\ref{x-y-xor}).

In particular, the self-similar measure $\pi_0$ on $T$ solves problem (\ref{x-y-xor})
with marginals $\mu=\nu$, where
$\mu $ can be described as the distribution of the series
$
\sum_{i=1}^{\infty} \frac{\xi_i}{2^i},
$
where the sequence  of i.i.d. Bernoulli random variables $\{\xi_i\}$ satisfies
$\xi_i = 1$ with probability $1/3$ and $\xi_i=0$ with probability $2/3$. Another example is the (normalized) Lebesgue measure on the main diagonal.
\end{remark}

\section{Main example. Dual problem.}

For the problem 
$$
\int xyz d \pi \to \min,
$$
where $\pi$ has Lebesgue projections onto principal hyperplane, let us consider the corresponding  dual problem:
\begin{equation}
\label{main-dual}
\int F(x,y) dx dy + \int G(y,z) dx dy + \int H(z,x) dx dz \to \max,
\end{equation}
\begin{equation}
\label{main-dual0}
F(x,y) + G(y,z)  + H(z,x) \le xyz.
\end{equation}
It is clear that by symmetries of the problem one can reduce the general problem to the case
$$
F=G=H, \ F(x,y) = F(y,x).
$$
Let us remind to the reader that by the standard duality arguments any function $F$ satisfying (\ref{main-dual0})
and
$$
 F(x,y)  +  G(y,z)  +  H(z,x) = xyz, \ z = x \oplus y 
$$
$(x,y)$-almost everywere is a solution to (\ref{main-dual}).

Discretizing the problem and performing  finite-dimensional linear programming algorithms we were able to guess
 reccurent relations for the restriction of $F$ onto the set of dyadic rational numbers. Using these relations we prove 
the desired properties of our function. Finally, we will give an integral representation for the solution in the next section.

\subsection{Definition and easy properties.}

Let $\mathbb{N}_0$ be the set of all non-negative integers.

\begin{definition}

Let $f: \mathbb{N}_0 \times \mathbb{N}_0 \to \mathbb{Z}$ be a function defined as follows. Set:
$$f(0, 0) = 0, f(0, 1) = f(1, 0) = -1, f(1, 1) = 2.$$ In all other points $f$ is defined by the following  recurrent relations:
\begin{align} \label{definition}
	f(a, b) = 
	\begin{cases}
	8f({a \over 2}, {b \over 2}) \ \text{if } a \equiv 0 \Mod 2 \text{ and } b \equiv 0 \Mod 2, \\
	4\left(f({a - 1 \over 2}, {b \over 2}) + f({a + 1 \over 2}, {b \over 2})\right) + 3 \ \text{if} \ a \equiv 1 \Mod 2 \text{ and } b \equiv 0 \Mod 2,\\
	4\left(f({a \over 2}, {b - 1 \over 2}) + f({a \over 2}, {b + 1 \over 2})\right) + 3\  \text{if } a \equiv 0 \Mod 2 \ \text{and } b \equiv 1\Mod 2, \\
	2\left(f({a - 1 \over 2}, {b - 1 \over 2}) + f({a - 1 \over 2}, {b + 1 \over 2}) + f({a + 1 \over 2}, {b - 1 \over 2}) + f({a + 1 \over 2}, {b + 1 \over 2})\right) + 2,& \\
	\text{if} \ a \equiv 1 \Mod 2 \text{ and} \ b \equiv 1\Mod 2.
	\end{cases}
\end{align} 

\end{definition}

The following properties can be immediately derived from the definition.
\begin{equation}
	f(a, b) = f(b, a).
\end{equation}
	If $a$ is odd and $b$ is even, then
	\begin{equation} \label{even_odd}
	f(a, b) = {1 \over 2}(f(a + 1, b) + f(a - 1, b)) + 3.
\end{equation}

	If $a$ is odd and $b$ is odd, then 
\begin{equation} \label{odd_odd}
f(a, b) = {1 \over 2}(f(a + 1, b) + f(a - 1, b)) - 2.
\end{equation}

\begin{equation} \label{parity}
	f(a, b) \equiv a + b \Mod 2.
\end{equation}

	\subsection{Continuity.}
	
	Using the homogeneity relation 
	$$f(2a, 2b) = 8f(a, b)$$
	with factor two one can define $f_C(x, y)$ for any non-negative binary-rational $x$ and $y$. Namely, assume that  $(x, y) = ({a \over 2^n}, {b \over 2 ^ n})$, then one can set $f_C(x, y) = 8^{-n}f(a, b)$. It  is easy to check that $f_C$ is well-defined. In
	what follows we  extend $f_C$ to all pairs of non-negative real numbers by continuity. 
	To this end we need some estimates of the  increments of $f$.
	
	Consider a family of integer segments $I_n$: $I_n = [0, 2^{n + 1}]$,  $n \ge 0$. Note that for any $a \in I_n$ with $n \ge 1$ the numbers ${a \over 2}$ for  even $a$, and ${a + 1 \over 2}$ and ${a - 1 \over 2}$ for odd  $a$, belong to the segment $I_{n - 1}$.
	
	Set: $$N_{n, m} = \max(|f(a + 1, b) - f(a, b)| : a, (a + 1) \in I_n, b \in I_m).$$ 
	
	\begin{lemma}
		There exists universal constant $C$, such that $N_{n, m} \le C(4^n + 4^m)$.
	\end{lemma}
	
	\begin{proof}
	It will be convenient to prove more general  inequality $N_{n, m} \le C_1(4^n + 4^m) + C_2$ applying induction method. At the end we obtain 
	that $C_2$ can take negative values. 
		
		Base of induction for $n = m = 0$ can be checked directly: $N_{0, 0} = 15 \le 2C_1 + C_2$.
		
		To prove the step of induction  let us estimate $|f(a + 2, b) - f(a, b)|$, where $b \in I_m$, $a, (a + 2) \in I_n$ and $a$ is even.

		Let $b$ be even. Then $|f(a + 2, b) - f(a, b)| = 8|f({a \over 2} + 1, {b \over 2}) - f({a \over 2}, {b \over 2})|$. If $n$ and $m$ are both strictly positive, we obtain by induction hypothesis 
		$$8\left|f\left({a \over 2} + 1, {b \over 2}\right) - f\left({a \over 2}, {b \over 2}\right)\right| \le N_{m-1, n-1}.$$
		If only one number (say, $m$) is positive,  then 
		$$8\left|f\left({a \over 2} + 1, {b \over 2}\right) - f\left({a \over 2}, {b \over 2}\right)\right| \le N_{m-1, 0}.$$ 
		In any case one gets
		\begin{multline}
		8\left|f\left({a \over 2} + 1, {b \over 2}\right) - f\left({a \over 2}, {b \over 2}\right)\right| \le 8\left(C_1\left(4^{n - 1} + 4^{m - 1} + {3 \over 4}\right) +C_2\right) = \\ = 2C_1(4^n + 4^m) + (6C_1 + 8C_2).$$
		\end{multline}
		Here we used  inequality $4^{\max(n - 1, 0)} + 4^{\max(m - 1, 0)} \le 4^{n - 1} + 4^{m - 1} + {3 \over 4}$, which holds
		provided one of the numbers  $n, m$ is positive.
		
		Using that $a + 1$ is odd and applying the recurrent relations \eqref{even_odd}  one gets
		\begin{align*}
		&f(a + 1, b) = {1 \over 2}(f(a, b) + f(a + 2, b)) + 3, \\
		&f(a + 2, b) - f(a + 1, b) = {1 \over 2}(f(a + 2, b) - f(a, b)) - 3, \\
		&f(a + 1, b) - f(a, b) = {1 \over 2}(f(a + 2, b) - f(a, b)) + 3.
		\end{align*}
		
		These estimates imply that $|f(a + 1, b) - f(a, b)|$ and $|f(a + 2, b) - f(a + 1, b)|$ can be estimated from above by
		\begin{multline}
		{1 \over 2}|f(a + 2, b) - f(a, b)| + 3 \le C_1(4^n + 4^m) + (3C_1 + 4C_2 + 3) \le  \\ \le C_1(4^n + 4^m) + C_2,
		\end{multline}
		provided $3C_1 + 4C_2 + 3 \le C_2$.
		
		Hence we obtain that for any even $b \in I_m$ and for any even $a, (a + 1) \in I_n$ the following inequality holds: $|f(a + 1, b) - f(a, b)| \le C_1(4^n + 4^m) + C_2$.
		
		Let now $b$ be odd. We estimate $|f(a + 2, b) - f(a, b)|$ for any even $a$ satisfying $a, (a + 2) \in I_n$ in a similar manner. Using recurrent relations \eqref{definition} we obtain:
		
		\begin{multline*}
			f(a + 2, b) - f(a, b) = 4\left[f\left({a \over 2} + 1, {b + 1 \over 2}\right) - f\left({a \over 2}, {b + 1 \over 2}\right)\right] + \\ 
			+ 4\left[f\left({a \over 2} + 1, {b - 1 \over 2}\right) - f\left({a \over 2}, {b - 1 \over 2}\right)\right] + 6 \le \\ \le 8\left[C_1\left(4^{n - 1} + 4^{m - 1} + {3 \over 4}\right) + C_2\right] + 6 = 2C_1(4^n + 4^m) + (6C_1 + 8C_2 + 6).
		\end{multline*}
		
		Next we  estimate $|f(a + 1, b) - f(a, b)|$ and $|f(a + 2, b) - f(a + 1, b)|$. Since $a + 1$ and $b$ are odd, one gets applying \eqref{odd_odd}  
		\begin{align*}
		&f(a + 1, b) = {1 \over 2}(f(a, b) + f(a + 2, b)) - 2, \\
		&f(a + 2, b) - f(a + 1, b) = {1 \over 2}(f(a + 2, b) - f(a, b)) + 2, \\
		&f(a + 1, b) - f(a, b) = {1 \over 2}(f(a + 2, b) - f(a, b)) - 2.
		\end{align*}
		
		Finally,
		\begin{multline*}
			|f(a + 1, b) - f(a, b)|, |f(a + 2, b) - f(a + 1, b)| \le \\
			\le {1 \over 2}|f(a + 2, b) - f(a, b)| + 2 \le C_1(4^n + 4^m) + 3C_1 + 4C_2 + 5 \le \\
			\le C_1(4^n + 4^m) + C_2,
		\end{multline*}
		provided that $3C_1 + 4C_2 + 5 \le C_2$.
		
		Now we get that for all odd $b \in I_m$ and for all $a, (a + 1) \in I_n$ one has 
		$$|f(a + 1, b) - f(a, b)| \le C_1(4^n + 4^m) + C_2.$$
		 This implies $N_{n, m} \le C_1(4^n + 4^m) + C_2$, which completes the induction step.
		
		To conclude it is sufficient to find  solutions $C_1$ and $C_2$ to the  following system of inequalities
		\begin{align}
		\begin{cases*}
			2C_1 + C_2 \ge 15, \\
			3C_1 + 4C_2 + 3 \le C_2, \\
			3C_1 + 4C_2 + 5 \le C_2.
		\end{cases*}
		\end{align}
		Set: $C_1 = 17, C_2 = -19$. This completes the proof.
	\end{proof}

	In what folows we consider the square $$I = [0, 2^{N + 1}] \times [0, 2^{N + 1}].$$ 
	Assume that  dyadic rational numbers $x, \Delta x, y, \Delta y$ satisfy $(x, y), (x + \Delta x, x + \Delta y) \in I$. 
	
	\begin{lemma}
		$|f_C(x + \Delta x, x + \Delta y)  - f_C(x, y)| \le 2^{2N + 1}C(|\Delta x| + |\Delta y|)$.
	\end{lemma}
	
	\begin{proof}
		There exist an integer number $M$, such that $2^Mx$, $2^My$, $2^M \Delta x$, $2^M \Delta y$ are non-negative integers. Then the desired result follows from the line of inequalities
		
		\begin{multline*}
		|f_C(x + \Delta x, x + \Delta y)  - f_C(x, y)| = \\ 
		= {1 \over 8^M}|f(2^M(x + \Delta x), 2^M(y + \Delta y)) - f(2^Mx, 2^My)| \le \\ \le {1 \over 8^M}2^M(|\Delta x| + |\Delta y|)N_{N + M, N + M} \le  \\ 
		\le{1 \over 4^M}C(4^{N + M} + 4^{N + M})(|\Delta x| + |\Delta y|) = 2^{2N + 1}C(|\Delta x| + |\Delta y|).
		\end{multline*}
	\end{proof}
	
This  statement immediately implies that for every  Cauchy sequence  $(x_i, y_i)$ the sequence $f_C(x_i, y_i)$ is a Cauchy sequence as well. Thus $f_C$ can be extended to a continuous function on the set of non-negative real numbers. In what follows $f_C$ denotes this extension.
	
	From the properties of $f$ and continuity of $f_C$ we infer the important  homogeneity property:
	
	\begin{proposition} \label{scalability}
		$$f_C(2x, 2y) = 8f_C(x, y).$$
	\end{proposition}
	
\subsection{Solution to the dual problem.}

In this section we prove our main duality result.
Namely, let us set 
$$F(a, b, c) = f(a, b) + f(b, c) + f(c, a)$$ and 
$$F_C(x, y, z) = f_C(x, y) + f_C(y, z) + f_C(z, x).$$
We show that function $\frac{1}{8} F_C$ solves the dual problem. Note that
Theorem \ref{dualth} does not establish existence of a solution to the dual problem.
In this concrete example we construct it explicitly.

The following theorem is the main result of this section.

\begin{theorem}
Function $F_C$ satisfies
$$
F_C(x,y,z) \le 8 xyz.
$$
The case of equality $F_C(x,y,z) = 8 xyz$ holds if and only if $(x,y,z)$ belongs to the closure of the set
$$
x \oplus y \oplus z=0.
$$
In particular, the triple $\frac{1}{8} f_C(x,y), \frac{1}{8} f_C(x,z), \frac{1}{8} f_C(y,z)$ solves
problem (\ref{main-dual}).
\end{theorem}
\begin{proof}
See Corollary \ref{1803} and Proposition \ref{xyz-equality}.
\end{proof}

\begin{proposition}
Function $F(a,b,c)$ satisfies inequality $$F(a, b, c) \le 8abc.$$ 
The equality case 
\begin{equation}
\label{equality}
F(a, b, c) = 8abc
\end{equation}
can hold only if $a + b + c \equiv 0 \Mod{2}$. 
\end{proposition}
In particular, continuity of $f_C$ implies
\begin{corollary} \label{1803}
 $$F_C(x, y, z) \le 8xyz.$$
\end{corollary}
\begin{proof}
Let us prove the claim by induction. Base of induction is easy to check.	
	Note that $F(a, b, c) = f(a, b) + f(b, c) + f(c, a) \equiv (a + b) + (b + c) + (c + a) \equiv 0 \Mod{2}$ because of \eqref{parity}. The latter implies  $F(a, b, c) \le 8abc - 2$ provided $F(a, b, c) < 8abc$. 
	
	To prove  the induction step we consider several cases.
	
	\begin{itemize}
	
	\item All of $a, b, c$ are even. From \eqref{definition} we infer $F(a, b, c) = 8F({a \over 2}, {b \over 2}, {c \over 2})$.  By induction hypothesis $F(a, b, c) = 8F({a \over 2}, {b \over 2}, {c \over 2}) \le 8\cdot8 \cdot {a \over 2}\cdot {b \over 2}\cdot {c \over 2} = 8abc$. 
	
	\item Assume that one of the numbers $a, b, c$ (say, $a$) is odd and the other are even.
	We need to check 
	\begin{equation}
	\label{8abc-2}
	F(a, b, c) \le 8abc - 2,
	\end{equation} because $a + b + c \equiv 1 \Mod{2}$. Applying \eqref{definition} one gets
	\begin{multline*}
	F(a, b, c) = f(a, b) + f(a, c) + f(b, c) = \\ 
	= \left[4\left(f\left({a - 1 \over 2}, {b \over 2}\right) + f\left({a + 1 \over 2}, {b \over 2}\right)\right) + 3\right] + \\ 	
	+\left[4\left(f\left({a - 1 \over 2}, {c \over 2}\right) + f\left({a + 1 \over 2}, {c \over 2}\right)\right) + 3\right] + 8f\left({b \over 2}, {c \over 2}\right) = \\ 
	= 4\left(F\left({a - 1 \over 2}, {b \over 2}, {c \over 2}\right) + F\left({a + 1 \over 2}, {b \over 2}, {c \over 2}\right)\right) + 6.
	\end{multline*}
	
	One of the triples  $\left({a - 1 \over 2}, {b \over 2}, {c \over 2}\right)$, $\left({a + 1 \over 2}, {b \over 2}, {c \over 2}\right)$ admits even sum of elements, hence satisfies (\ref{8abc-2}).
	
	Therefore we can write:
	\begin{multline}
	4\left(F\left({a - 1 \over 2}, {b \over 2}, {c \over 2}\right) + F\left({a + 1 \over 2}, {b \over 2}, {c \over 2}\right)\right) + 6 \le \\ 4((a - 1)bc + (a + 1)bc - 2) + 6 = 8abc - 2.
	\end{multline}
	
	\item Assume that there are exactly two odd  numbers among  $a, b, c$. Without loss of generality they are $a$ and $b$. Check that $F(a, b, c) \le 8abc$, because $a + b + c \equiv 0 \Mod{2}$. Applying \eqref{definition} one gets
	\begin{multline}
	F(a, b, c) = f(a, b) + f(b, c) + f(c, a) = \\ 
	= \left[2\sum_{\Delta a, \Delta b \in \{-1, 1\}}f\left({a + \Delta a \over 2}, {b + \Delta b \over 2}\right) + 2\right] + \\ 
	+ \left[4\sum_{\Delta b \in \{-1, 1\}}f\left({b + \Delta b \over 2}, {c \over 2}\right) + 3\right] + \\ + \left[4\sum_{\Delta a \in \{-1, 1\}}f\left({a + \Delta a \over 2}, {c \over 2}\right) + 3\right] = \\ 
	= 2\sum_{\Delta a, \Delta b \in \{-1, 1\}}F\left({a + \Delta a \over 2}, {b + \Delta b \over 2}, {c \over 2}\right) + 8.
	\end{multline}
	
	Note that triples of the type $({a + \Delta a \over 2}, {b + \Delta b \over 2}, {c \over 2})$ there are exactly two with even sum of elements, so by induction hypothesis for at most two triples  (\ref{equality}) holds.
	
	Hence
	\begin{align*}
	2 &\sum_{\Delta a, \Delta b \in \{-1, 1\}}F\left({a + \Delta a \over 2}, {b + \Delta b \over 2}, {c \over 2}\right) + 8 \le  
	 2((2a)(2b)c - 2 \cdot 2) + 8  = 8abc.
	\end{align*}
	
	\item Finally let us assume that all $a, b, c$ are odd. Thus $a + b + c \equiv 1 \Mod{2}$, so we need to check $F(a, b, c) \le 8abc - 2$. Again, \eqref{definition} implies
	\begin{multline*}
	F(a, b, c) = f(a, b) + f(b, c) + f(c, a) = \\ 
	= \left[2\sum_{\Delta a, \Delta b \in \{-1, 1\}}f\left({a + \Delta a \over 2}, {b + \Delta b \over 2}\right) + 2\right] + \\ 
	+ \left[2\sum_{\Delta b, \Delta c \in \{-1, 1\}}f\left({b + \Delta b \over 2}, {c + \Delta c \over 2}\right) + 2\right] + \\ 
	+ \left[2\sum_{\Delta c, \Delta a \in \{-1, 1\}}f\left({c + \Delta c \over 2}, {a + \Delta a \over 2}\right) + 2\right]  = \\ 
	=\sum_{\Delta a, \Delta b, \Delta c \in \{-1, 1\}}F\left({a + \Delta a \over 2}, {b + \Delta b \over 2}, {c + \Delta c \over 2}\right) + 6.
	\end{multline*}
	Counting the equality cases and repeating the arguments from above one gets
	\begin{multline*}
	\sum_{\Delta a, \Delta b, \Delta c \in \{-1, 1\}}F\left({a + \Delta a \over 2}, {b + \Delta b \over 2}, {c + \Delta c \over 2}\right) + 6 \le \\ 
	\le (2a)(2b)(2c) - 2 \cdot 4 + 6 = 8abc - 2.
	\end{multline*}
	\end{itemize}

	Step of induction is verified in all possible cases.
	
\end{proof}

\subsection{Some nice identities.}

Here we prove  some other useful identities for $f(a, b)$ and their continious analogues for $f_C(x, y)$.

\begin{proposition}
\label{1803pr}
	Let $0 \le a, b \le 2^n$. Then $$f(2^n + a, 2^n + b) = 2 \cdot 8^n + 6 \cdot 4^n(a + b) + f(a, b).$$
\end{proposition}

\begin{proof}
	Apply induction by $n$. The case  $n = 0$ is easy to check. 
	
	Let $a$ and $b$ be even. Then
	\begin{multline*}
	f(2^n + a, 2^n + b) = 8f\left(2^{n - 1} + {a \over 2}, 2^{n - 1} + {b \over 2}\right) = \\ 
	= 8\left(2 \cdot 8^{n - 1} + 6 \cdot 4^{n - 1}{a + b \over 2} + f\left({a \over 2}, {b \over 2}\right)\right) = \\ 
	= 2 \cdot 8^n + 6 \cdot 4^n(a + b) + f(a, b).
	\end{multline*}
	
	Let exactly one of the numbers  $a$ or $b$ be odd. Without loss of generality assume that $a$ is odd. 
	Then
	
	\begin{multline*}
	f(2^n + a, 2^n + b) = \\
	= 4\left(f\left(2^{n - 1} + {a + 1 \over 2}, 2^{n - 1} + {b \over 2}\right) + f\left(2^{n - 1} + 2^{n - 1} + {a - 1 \over 2}, {b \over 2}\right)\right) + 3 = \\ 
	= 4\left[2\cdot 8^{n - 1} + 6 \cdot 4^{n - 1}{a + b + 1 \over 2} + f\left({a + 1 \over 2}, {b \over 2}\right)\right] + \\
	+ 4\left[2\cdot 8^{n - 1} + 6 \cdot 4^{n - 1}{a + b - 1 \over 2} + f\left({a - 1 \over 2}, {b \over 2}\right)\right] + \\ 
	+ 4\left[f\left({a + 1 \over 2}, {b \over 2}\right) + f\left({a - 1 \over 2}, {b \over 2}\right) \right] + 3 =  \\
	= 2 \cdot 8^n + 6 \cdot 4^n(a + b) + f(a, b).
	\end{multline*}
	
	Let both of $a$ and $b$ be odd. Similarly by the induction hypothesis:
	
	\begin{multline*}
	f(2^n + a, 2^n + b) = 2\sum_{\Delta a, \Delta b \in \{-1, 1\}}f\left(2^{n - 1} + {a + \Delta a \over 2}, 2^{n - 1} + {b + \Delta b \over 2}\right) + 2 = \\
	= 2\sum_{\Delta a, \Delta b \in \{-1, 1\}}\left[2 \cdot 8^{n - 1} + 6 \cdot 4^{n - 1}{a + b + \Delta a + \Delta b \over 2} + f\left({a + \Delta a \over 2 }, {b + \Delta b \over 2}\right)\right] + 2 = \\
	= 2 \cdot 8^n + 6 \cdot 4^n(a + b) + 2\sum_{\Delta a, \Delta b \in \{-1, 1\}}f\left({a + \Delta a \over 2 }, {b + \Delta b \over 2}\right) + 2 = \\
	= 2 \cdot 8^n + 6 \cdot 4^n(a + b) + f(a, b).
	\end{multline*}
\end{proof}

\begin{proposition}
	For $0 \le a, b \le 2^n$ one has $$f(2^n + a, b) = -8^n - 6~\cdot~4^na + 4^{n + 1}b + 8\cdot2^nab + f(a, b).$$
\end{proposition}

\begin{proof}
	The proof is similar to the proof of Proposition \ref{1803pr}. Apply the induction by $n$. The base for $n = 0$ can be checked by an easy computation. 
	
	Consider $n \ge 1$.
	Set 
	$g_1(n, a, b) = -8^n$, $ g_2(n,a,b) =  6~\cdot~4^na$, $g_3(n,a,b) = 4^{n + 1}b$, $g_4(n,a,b)= 8\cdot2^nab$.
	 For induction step it is sufficient to check that the following identities hold:
	\begin{align*}
	&g_i(n, a, b) = 8g_i\left(n - 1, {a \over 2}, {b \over 2}\right), \\
	&g_i(n, a, b) = 4g_i\left(n - 1, {a - 1 \over 2}, {b \over 2}\right) + 4g_i\left(n - 1, {a + 1 \over 2}, {b \over 2}\right), \\
	&g_i(n, a, b) = 4g_i\left(n - 1, {a \over 2}, {b - 1 \over 2}\right) + 4g_i\left(n - 1, {a \over 2}, {b + 1 \over 2}\right),\\
	&g_i(n, a, b) = 2\sum_{\Delta a, \Delta b \in \{-1, 1\}}g_i\left(n - 1, {a + \Delta a \over 2}, {b + \Delta b \over 2}\right).
	\end{align*}
	
	Next we prove the desired identity by considering four different cases: $a$ is odd(even), $b$ is odd(even)
	and applying an appropriate  identity for all summands in the right hand side. For any of $8^n$, $4^na$, $4^nb$ and $2^nab$ these properties are obviously true. 
\end{proof}

Clearly, the continious analogues of these identities look as follows.
\begin{proposition} \label{nice}
	Let $0 \le x, y \le {1 \over 2}$. Then:
	\begin{align*}
	&f_C\left({1 \over 2} + x, {1 \over 2} + y\right) = {1 \over 4} + {3 \over 2}(x + y) + f_C(x, y), \\
	&f_C\left({1 \over 2} + x, y\right) = -{1 \over 8} - {3 \over 2}x + y + 4xy + f_C(x, y).
	\end{align*}
\end{proposition}

\subsection{Case of equality}

	\begin{proposition}
	\label{xyz-equality}
	The relation $x \oplus y \oplus z = 0$  implies $F_C(x, y, z) = 8xyz$. 
\end{proposition}
	\begin{proof}
	Assume the opposite and consider the maximum of $8xyz - F_C(x, y, z)$  on the closure $S$ of the  set of points  $(x, y, z)$ satisfying $x \oplus y \oplus z = 0$. This maximum $C$ exists since the set is compact and  $8xyz - F_C(x, y, z)$ is continuous. 
	It is sufficient to show that $C$ is not strictly positive.
	Find a point $(x_0, y_0, z_0)$ with $z_0 = x_0 \oplus y_0$ such that $8x_0 y_0 z_0 - F_C(x_0, y_0, z_0) > C/2$.
	
	The first numbers in the binary representations of $x_0,y_0,z_0$ contains either all zeroes of exactly two units, because  $x_0 \oplus y_0 \oplus z_0 = 0$. If they all are zeroes,
 then $2x_0 \oplus 2y_0 \oplus 2z_0 = 0$. Thus $8(2x_0)(2y_0)(2z_0) - F_C(2x_0, 2y_0, 2z_0) > 4C > C$, this contradicts to the 
 choice of $C$. If the numbers contain two units, without loss of generality assume $x_0=y_0=1$.
 Set $x_0 = {1 \over 2} + x_1$, $y_0 = {1 \over 2} + y_1$. The identities  \eqref{nice} imply
	
	\begin{multline*}
	8x_0y_0z_0 - F_C(x_0, y_0, z_0) = 8\bigl(x_1 + {1 \over 2}\bigl)\bigl(y_1 + {1 \over 2}\bigr)z_0 - F_C\bigl(x_1 + {1 \over 2}, y_1 + {1 \over 2}, z_0\bigr) = \\ 
	= 8x_1y_1z_0 + 4x_1z_0 + 4y_1z_0 + 2z_0-\\
	- f_C\left(x_1 + {1 \over 2}, y_1 + {1 \over 2}\right) - f_C\left(x_1 + {1 \over 2}, z_0\right) 
	- f_C\left(y_1 + {1 \over 2}, z_0\right) = \\
	= 8x_1y_1z_0 + 4x_1z_0 + 4y_1z_0 + 2z_0
	- \left[{1 \over 4} + {3 \over 2}(x_1 + y_1) + f_C(x_1, y_1)\right] - \\ - \left[-{1 \over 8} - {3 \over 2}x_1 + z_0 + 4x_1z_0 + f_C(x_1, z_0)\right]
	- \left[-{1 \over 8} - {3 \over 2}y_1 + z_0 + 4y_1z_0 + f_C(y_1, z_0)\right] = \\ 
	= 8x_1y_1z_0 - F_C(x_1, y_1, z_0).
	\end{multline*}
	Note that $x_1 \oplus y_1 \oplus z_0 = 0$, moreover, the function  $8xyz - F_C(x, y, z)$  takes at the point $(x_1, y_1, z_0)$ the same value $C/2$. Note that $x_1, y_1, z_0 \le {1 \over 2}$, but we have already shown that this is impossible. We got a contradiction.
	\end{proof}

\section{Integral representation of $f_C(a, b)$.}
	
	The solution to the dual problem in our main example has a simple relation to the
	(cumulative) distribution function
	$$
	I(a, b) = \int_{0}^{a} \int_{0}^{b}x \oplus y~dy dx, \ a, b \in \mathbb{R}_{+}.
	$$
	of the measure  $x \oplus y \ dx dy$.
	This function admits the following properties:
	
	\begin{property}
		Symmetry:  $I(a, b) = I(b, a)$.
	\end{property}
	
	\begin{property}
		Homogeneity with respect to factor $2$: $I(2a, 2b) = 8I(a, b)$.
	\end{property}
	
	\begin{proof}
		Note that for almost all $x$, $y$ and integer number $n$ one has $2^nx \oplus 2^ny = 2^n(x \oplus y)$. This yields
		\begin{multline*}
		I(2a, 2b) = \int_{0}^{2a}\int_{0}^{2b}x \oplus y~dy dx = {x = 2u \brack y = 2v} =\\
		= 4\int_{0}^{a}\int_{0}^{b} 2u \oplus 2v~dvdu = 8I(a, b).
		\end{multline*}
	\end{proof}
	
	\begin{property}
		For all $0 \le a \le 1$
		$$I(a, 1)~=~{a \over 2}.$$
	\end{property}
	
	\begin{proof}
		To this end we need the following lemma:
		
		\begin{lemma}
			For every couple $0 \le x, y \le 1$, where neither $x$ nor $y$ is binary rational, the following relation holds:
			 $x \oplus y + x \oplus (1 - y) = 1$.
		\end{lemma}
		
		\begin{proof}			
		Note that for  $a=x \oplus y$ and $b=x \oplus (1 - y)$ 
		the $i$-th digits satisfy $a_i = x_i \oplus y_i$, $b_i = x_i \oplus \overline{y_i}$. Clearly, $a_i \oplus b_i =0$.
		\end{proof}
		
		This can be used for computation of  $I(a, 1)$:
		\begin{multline*}
		I(a, 1) = \int_{0}^{a}\int_{0}^{1}x\oplus y~dydx = \\
		= {1 \over 2}\int_{0}^{a}\int_{0}^{1}\left(x\oplus y + x \oplus (1 - y)\right)~dydx = {1 \over 2}\int_{0}^{a}\int_{0}^{1} 1~dydx = {a \over 2}.
		\end{multline*}
	\end{proof}
	
	Applying homogeneity property one immediately gets
	\begin{corollary}
		For every  $0 \le a \le {1 \over 2^n}$  
		$$I\left(a, {1 \over 2^n}\right) = {a \over 2^{2n + 1}}.$$
	\end{corollary}

	In the following proposition we establish a recurent relation for  $f_C$:
	
	\begin{proposition}
		For all $0 \le a, b \le {1 \over 2}$ the following identity holds: 
		$$I\big({1 \over 2} + a, b \big) = {1 \over 2}ab  + {1 \over 8}b + I(a, b).$$
	\end{proposition}
	
	\begin{proof}
	Represent the integral as a sum of two parts
		\begin{multline*}
			I\left({1 \over 2} + a, b\right) = \int_{0}^{{1 \over 2} + a}\int_{0}^{b} x\oplus y~dydx = \\
			= \int_{1 \over 2}^{{1 \over 2} + a}\int_{0}^{b} x\oplus y~dydx + \int_{0}^{1 \over 2}\int_{0}^{b} x\oplus y~dydx.
		\end{multline*}
		
		Making the change of variable  $x = {1 \over 2} + t$ one gets
		\begin{multline*}
		\int_{1 \over 2}^{{1 \over 2} + a}\int_{0}^{b} x\oplus y~dydx = \int_{0}^{a}\int_{0}^{b}\left({1 \over 2} + t\right)\oplus y~dydt = \\
		= \int_{0}^{a}\int_{0}^{b}\left(t\oplus y + {1 \over 2}\right)dydt = {1 \over 2}ab + I(a, b).
		\end{multline*}
		
		Hence
		$$
		I(a, b) = {1 \over 2}ab + I(a, b) + I\left({1 \over 2}, b\right) = {1 \over 2}ab  + {1 \over 8}b + I(a, b).
		$$
	\end{proof}
	
	Let us prove another similar relation
	\begin{proposition}
		For every  $0 \le a, b \le {1 \over 2}$ one has 
		$$I\bigl({1 \over 2} + a, {1 \over 2} + b\bigr) = {1 \over 16} + {3 \over 8}a + {3 \over 8}b + I(a, b).$$
	\end{proposition}
	\begin{proof}
		Similarly to the arguments of the previous proposition one obtains
		\begin{multline*}
			I\left({1 \over 2} + a, {1 \over 2} + b\right) = \int_{0}^{{1 \over 2} + a}\int_{0}^{{1 \over 2} + b} x\oplus y~dydx =  \\
			= \int_{0}^{1 \over 2}\int_{0}^{1 \over 2} x\oplus y~dydx + \int_{1 \over 2}^{{1 \over 2} + a}\int_{0}^{1 \over 2} x\oplus y~dydx + \\
			+ \int_{0}^{1 \over 2}\int_{1 \over 2}^{{1 \over 2} + b} x\oplus y~dydx + \int_{1 \over 2}^{{1 \over 2} + a}\int_{1 \over 2}^{{1 \over 2} + b} x\oplus y~dydx.
		\end{multline*}
		
		Clearly
		$$
		\int_{0}^{1 \over 2}\int_{0}^{1 \over 2} x\oplus y~dydx = I\left({1 \over 2}, {1 \over 2}\right) = {1 \over 16}.
		$$
		To compute the second integral let us make the variables change $x = {1 \over 2} + t$:
		\begin{multline*}
			\int_{1 \over 2}^{{1 \over 2} + a}\int_{0}^{1 \over 2} x\oplus y~dydx =  \int_{0}^{a}\int_{0}^{1 \over 2}\left({1 \over 2} + t\right)\oplus y~dydt = \\
			 = \int_{0}^{a}\int_{0}^{1 \over 2}\left({1 \over 2} + t\oplus y\right)~dydt = {1 \over 4}a + I\left({1 \over 2}, a\right) = \\ 
			 = {1 \over 4}a + {1 \over 8}a = {3 \over 8}a.
		\end{multline*}
		In the same way one gets the following formula for the third integral:
		$$
		\int_{0}^{1 \over 2}\int_{1 \over 2}^{{1 \over 2} + b} x\oplus y~dydx = {3 \over 8}b.
		$$
		
		To compute the last integral, let us set $x = {1 \over 2} + t$, $y = {1 \over 2} + u$:
		\begin{multline*}
			\int_{1 \over 2}^{{1 \over 2} + a}\int_{1 \over 2}^{{1 \over 2} + b} x\oplus y~dydx = \int_{0}^{a}\int_{0}^{b} \left(t + {1 \over 2}\right)\oplus \left(u + {1 \over 2}\right)~dudt = \\
			 = \int_{0}^{a}\int_{0}^{b} t\oplus u~dudt = I(a, b).
		\end{multline*}
		
		Finally,
		$$
		I\left({1 \over 2} + a, {1 \over 2} + b\right) = {1 \over 16} + {3 \over 8}(a + b) + I(a, b).
		$$
	\end{proof}
	
	It remains to relate  $f_C$ and $I$.
	
	\begin{theorem}
		For all non-negative $x, y \in \mathbb{R}_{+}$ the following relation holds: $$f_C(x, y) = 8I(x, y) - 2I(x, x) - 2I(y, y).$$
	\end{theorem}
	
	\begin{proof}
		By homogeneity $f_C(x, y)$ and $I(x, y)$ it is sufficent to prove this relation on $[0, 1]^2$. 
		
		Set $f_1(x, y) = 8I(x, y) - 2I(x, x) - 2I(y, y)$. We prove that $f_1$ satifies the same relation as $f_C$ (see Proposition \ref{nice}). Indeed, for all, $0 \le x, y \le {1 \over 2}$:
		\begin{multline*}
		f_1\left({1 \over 2} + x, y\right) = 8I\left({1 \over 2} + x, y\right) - 2I\left({1 \over 2} + x, {1 \over 2} + x\right) - 2I(y, y) = \\
		 = 4xy + y + 8I(x, y) - 2\left({1 \over 16} + {3 \over 8}(x + x) + I(x, x)\right) - 2I(y, y) = \\
		 = -{1 \over 8} - {3 \over 2}x + y + 4xy+ f_1(x, y),
		\end{multline*}
		
		\begin{multline*}
		f_1\left({1 \over 2} + x, {1 \over 2} + y\right) = \\ =  8I\left({1 \over 2} + x, {1 \over 2} + y\right) - 2I\left({1 \over 2} + x, {1 \over 2} + x\right) - 2I\left({1 \over 2} + y, {1 \over 2} + y\right) = \\
		= {1 \over 2} + 3x + 3y + 8I(x, y) -  2\left({1 \over 16} + {3 \over 8}(x + x) + I(x, x)\right) -\\ - 2\left({1 \over 16} + {3 \over 8}(y + y) + I(y, y)\right)
		= {1 \over 4} + {3 \over 2}(x + y) + f_1(x, y).
		\end{multline*}
		
		It remains to show  that $M = \sup_{0 \le x \le 1, 0 \le y \le 1}|f - f_1|=0$. Note that the supremum is attained on  $\left[0, {1 \over 2}\right]^2$, because $f - f_1$ 
		is invariant with respect to the shifts $x \to x + \frac{1}{2}$, $y \to y + \frac{1}{2}$.
		If $M$ is larger than zero and attained at some point $(x_0, y_0)$, where $0 \le x_0, y_0 \le {1 \over 2}$, then the value of $|f - f_1|$ at  $(2x_0, 2y_0)$ equals $8M$.
		We obtained a contradiction.
	\end{proof}
	
	Applying the above result we obtain the following integral representation theorem 
	for our solution to the dual problem.
	
	\begin{theorem}
	The function 
$$
F(x,y) = \int_{0}^x \int_{0}^y s \oplus t \ ds dt - \frac{1}{4} \int_{0}^x \int_{0}^x s \oplus t \ ds dt
- \frac{1}{4} \int_{0}^y \int_{0}^y s \oplus t \ ds dt
$$
solves the dual  problem 
	$$
\int_{[0,1]^2} F(x,y) dx dy  + \int_{[0,1]^2} F(x,z) dx dz + \int_{[0,1]^2} F(y,z) dy dz  \to \max,
$$
$$F(x,y) + F(x,z) + F(y,z) \le xyz$$
to the primal $(3,2)$-Kantorovich  problem 
	$$
	\int xyz d \pi \to \min, \ (x,y,z) \in [0,1]^3,
	$$
	considered on the space of measure which projections onto principal hyperplanes are Lebesgue measures on $[0,1]^2$.
	\end{theorem}
	
	\section{Concluding remarks}

	Numerical experiments  visually reveal fractal structure of the solutions
	to (3,2)-Kantorovich problem for other cost functions and projections.
	This happens even under absence of symmetry, which, in turn, means
	that the solutions do not posess dyadic structure.
	Which properties of our main example are preserved in general case?
	Here we discuss several natural hypotheses. 
	
		\begin{question}
		\label{q1}
		Consider the $(3,2)$-Kantorovich problem on the set $X \times Y \times Z$, where 
		$$X = \{x_0 < x_1 \ldots < x_{2^n-1}\},$$
		 $$Y = \{y_0 < y_1 \ldots < y_{2^n-1}\},$$
		$$Z = \{z_0 <  z_1 \ldots < z_{2^n-1}\}.$$
		As usual, $c = xyz$ and the projections are supposed to be uniform.
		We want to maximize $\int xyz d \pi$. 
		
		Is it true that unifrom measure concentrated 
		on the points $(x_i, y_j, z_k)$ with $i \oplus j \oplus k = 0$ is optimal?
		\end{question}

	\begin{question}
	\label{q2}
Consider the dual  $(3,2)$-Kantorovich problem   on the set $[0,1]^3$.
$$
\int F(x,y) d \mu_{xy} + \int G(x,z) d \mu_{xz} +  \int H(y,z) d \mu_{yz} \to \max,
$$
$$
F(x,y) + G(x,z) + H(y,z) \le xyz
$$
for some triple of measures $\mu_{xy}, \mu_{xz}, \mu_{yz}$.

Is it true that $F$ satisfies inequality
$$
F(x + \Delta x, y + \Delta y) + 
F(x, y) - 
F(x + \Delta x, y) - 
F(x, y+ \Delta y) \ge 0
$$
for every $x, y, \Delta x \ge 0, \Delta y \ge 0$? Equivalently, $F$ has the representation
$$
F(x,y) = m([0,~x] \times [0,~y] )+ f(x) + g(y)
$$
for some nonnegative measure $m$ and some functions $f, g$?
\end{question}

	Numerical computations demonstrate that Question \ref{q2} has a negative answer.
	The answer to Question \ref{q1} is negative in general, but remarkably the answer is affirmative for $n=2$.

	 	\begin{example}
	 	Consider the discrete  cube $8\times 8\times 8$,
	 	$$X=Y=Z=\{0, \eps, 2\eps, 1-4\eps, 1-3\eps, 1-2\eps, 1-\eps, 1\}.$$ For sufficiently small  $\eps$, the uniform  measure $M'$, concentrated on the points $(x_i, y_j, z_k)$ with $i \oplus j \oplus k = 0$, $i, j, k \in \{0,1, \ldots, 2^3-1\}$,  is not optimal.
	 	Let us say that numbers $0, 1, 2$ are \textbf{small}. Other numbers are \textbf{large}.
	 	Consider the following  competitor: measure $M''$ assigns to a point $(x_i, y_j, z_k)$ the following value :
		$$\begin{cases}
			\frac{1}{3},&\text{if all three indexes $i$, $j$ and $k$ are small;}\\
		 	0,&\text{if two indexes are small and one is large;}\\
		 	\frac{1}{5},&\text{if one index is small and two are large;}\\
		 	\frac{2}{25},&\text{if all three indexes $i$, $j$ and $k$ are large.}
		\end{cases}
		$$
		 Integrals $\int xyz dM'$ and $\int xyz dM''$ are the polynomials in $\eps$. Their free terms are equal to $12$ and $125 \times \frac{2}{25} = 10$ respectively. Thus $\int xyz dM'>\int xyz dM''$ for sufficiently small epsilon. 		
\end{example}

Let $I = [0, 1]^3$ be the unit cube  and  $\mu$ be arbitrary measure on  $I$. 
	We denote by $F_{\mu}$ the  distribution function of $\mu$
	 $$F_\mu(a, b, c) = \mu([0, a] \times [0, b] \times [0, c]).$$

	\begin{lemma}
		Let  $\mu$ be a measure on  $I$. Then the following identity holds:
		
		$$
		\int\limits_{I}(1 - x)(1 - y)(1 - z)d\mu = \int\limits_{I}F_\mu(x, y, z)dxdydz.
		$$
	\end{lemma}
	
	\begin{proof}
		Let  $I'$ be the unit cube endowed with the uniform Lebesgue measure $\omega$. 
		One can consider the product $I \times I'$ with the product measure $d\mu \otimes d\omega$.
		Set: $$D = \{(p, q) \in I \times I' \mid \text{$p$ is not larger than $q$ coordinatewise}\}.$$ 
		
		Let us find $(\mu\otimes\omega)(D)$. We apply to this end the Fubini theorem
		$$
		\int\limits_D d\mu\otimes d\omega = \int\limits_{(x, y, z) \in I} \int\limits_{\substack
			{(x_1, y_1, z_1) \in I', \\ (x_1, y_1, z_1) \ge (x, y, z)}} d\omega d\mu = \int\limits_{(x, y, z) \in I}(1 - x)(1 - y)(1 - z)d\mu.
		$$
		
		On the other hand,
		$$
		\int\limits_{(x, y, z) \in I} = \int\limits_{(x, y, z) \in I'}\int\limits_{\substack
			{(x_1, y_1, z_1) \in I, \\ (x_1, y_1, z_1) \le (x, y, z)}} d\mu d\omega = \int\limits_{(x, y, z) \in I}F_\mu(x, y, z)d\omega.
		$$
	\end{proof}
	
	Let  $\mu_{xy}, \mu_{yz}, \mu_{zx}$ be projections of  $\mu$ onto the corresponding  principal hyperplanes.
	On can rewrite the integral as follows:
	
	\begin{multline*}
	\int\limits_I(1 - x)(1 - y)(1 - z)d\mu = 1 - \int\limits_{I_{xy}}xy~d\mu_{xy} - \int\limits_{I_{yz}}yz~d\mu_{yz} - \int\limits_{I_{zx}}zx~d\mu_{zx} + \\
	+ \int\limits_{I_{xy}}x~d\mu_{xy} + \int\limits_{I_{yz}}y~d\mu_{yz} + \int\limits_{I_{zx}}z~d\mu_{zx} - \int\limits_{I}xyz~d\mu = \\
	= C(\mu_{xy}, \mu_{yz}, \mu_{zx}) - \int\limits_{I}xyz~d\mu,
	\end{multline*}
	where $C(\mu_{xy}, \mu_{yz}, \mu_{zx})$ only depends on the projections of  $\mu$ onto the principal hyperplanes.

	We want to find a measure  $\pi$  which minimizes  $\int\limits xyz d\pi$ on the set of  all $(3, 2)$-stochastic measures on  $X \times Y \times Z$.

	Finally, consider  $$X = \{x_0 < x_1 \ldots < x_{2^n-1}\},$$ 
	$$Y = \{y_0 < y_1 \ldots < y_{2^n-1}\},$$
	$$Z = \{z_0 < z_1 \ldots < z_{2^n-1}\}.$$
	Without loss of generality assume that $X \times Y \times Z \subset I$.
		Let  $\mu_\oplus$ be a measure on $I$ which is supported on  $X \times Y \times Z$
		and defined by $$\mu_\oplus(x_i, y_j, z_k) = {1 \over 4^n},$$ if $i \oplus j \oplus k = 0$, and  $$\mu_\oplus(x_i, y_j, z_k)  =0$$ in the opposite case.

	\begin{theorem} Assume that $|X| = |Y| = |Z| = 4$. 
		Let  $\mu$ be arbitrary  measure  $X \times Y \times Z$
		with uniform projections on $X \times Y, X \times Z, Y \times Z$.
		Then $$\int\limits xyz~d\mu \ge \int\limits xyz~d\mu_\oplus.$$
		Moreover,
		$$
		F_{\mu_\oplus} \ge F_\mu
		$$
		at every point.
	\end{theorem} 
	
\begin{proof}Since the projections of  $\mu$ and $\mu_\oplus$ onto the hyperplanes  are equal, one has the following equivalence relation 
		$$
		\int\limits_I xyz~d\mu \ge \int\limits_I xyz~d\mu_\oplus \Leftrightarrow \int\limits_IF_\mu(x, y, z)~dxdydz \le \int\limits_IF_{\mu_\oplus}(x, y, z)~dxdydz.
		$$
		
		Let us prove that $F_{\mu_\oplus} \ge F_\mu$. Since the measures are discrete,
		it is suficient to check the desired inequality at the points $(x_i, y_j, z_k) \in X \times Y \times Z$. 
		Without loss of generality let $i \le j \le k$.
		
		If  $k = 3$, the distribution function satisfies  $F_{\mu_\oplus}(x_i, y_j, z_3) = F_{\mu}(x_i, y_j, z_3) = {(i + 1)(j + 1) \over 16}$. This follows from the fact that  $\mu$ and $\mu_\oplus$ have uniform projections onto  $X \times Y$.

		Let $i = 0$. Then $F_{\mu_\oplus}(x_0, y_j, z_k) = {1 \over 16}\min(j + 1, k + 1) = {j + 1 \over 16}$. 
		Indeed, for  $i = 0$ measure $\mu_\oplus$ is concentrated  at the points $(x_0, y_t, z_t)$, $t \in \{0,1,2,3\}$. Hence $F_{\mu_\oplus}(x_0, y_j, z_k) = {1 \over 16}\#(t \mid 0 \le t \le j, 0 \le t \le k)$. On the other hand $F_\mu(x_0, y_j, z_k) \le F_\mu(x_0, y_j, z_3) = {j + 1 \over 16}$.
		
		It remains to consider the cases when every $i, j, k$ equals $1$ or $2$.
		
		Let $k = 2$.  Compute $F_{\mu_\oplus}(x_i, y_j, z_2)$. To this end we count all triples  $(a, b, c)$ satisfying $0 \le a \le i$, $0 \le b \le j$, $0 \le c \le 2$ and $a \oplus b \oplus c = 0$. For every couple  $(a, b)$ there exists the unique   $c$ having this property except for the case  $a \oplus b = 3$. This happens if and only if $\{a, b\} = \{1, 2\}$. 
		It is easy to check that amount of couples with this property is exactly the number of indices $i,j$
		which takes value $2$, i.e. $i+j-2$.
		Thus the total amount of such triples  $(a, b, c)$ equals $(i + 1)(j + 1) - i - j + 2 = ij + 3$. Hence $F_{\mu_\oplus}(x_i, y_j, z_2) = {ij + 3 \over 16}$.
		
		Represent the number $F_\mu(x_i, y_j, z_2)$ as follows:
		
		\begin{multline*}
		F_\mu(x_i, y_j, z_2) = \sum\limits_{\substack{x \in [0, x_i] \\ y \in [0, y_j]\\ z \in [0, z_2]}} \mu(x, y, z) = \sum\limits_{\substack{x \in [0, x_i] \\ y \in [0, y_j]\\ z \in [0, z_3]}} \mu(x, y, z)  - \sum\limits_{\substack{x \in [0, x_i] \\ y \in [0, y_j]}}\mu(x, y, z_3) = \\
		= F_{\mu}(x_i, y_j, z_3) - \sum\limits_{\substack{x \in [0, x_i] \\ y \in [0, y_3]}}\mu(x, y, z_3) + \sum\limits_{\substack{x \in [0, x_i] \\ y \in [y_{j + 1}, y_3]}}\mu(x, y, z_3),
		\end{multline*}
		where the sum is taken over the atoms of $\mu$.
		
		We know that $F_\mu(x_i, y_j, z_3) = {(i + 1)(j + 1) \over 16}$, because the projection of $\mu$ onto $X \times Y$ 
		is uniform.
		Analogously, the same facts about projections onto $X \times Z$ and $Y \times Z$
		imply
		$$
		\sum\limits_{\substack{x \in [0, x_i] \\ y \in [0, y_3]}}\mu(x, y, z_3) = {i + 1 \over 16}.
		$$
		$$
		\sum\limits_{\substack{x \in [0, x_i] \\ y \in [y_{j + 1}, y_3]}}\mu(x, y, z_3) \le \sum\limits_{\substack{x \in [0, x_3] \\ y \in [y_{j + 1}, y_3]}}\mu(x, y, z_3) = {3 - j \over 16},
		$$
		Hence
		$$
		F_\mu(x_i, y_j, z_2) \le {(i + 1)(j + 1) \over 16} - {i + 1 \over 16} + {3 - j \over 16} = {ij + 3 \over 16}.
		$$
		
		It remains to consider the case  $i = j = k = 1$. One gets immediately $F_{\mu_\oplus}(x_i, y_j, z_k) = {4 \over 16}$, and $F_\mu(x_i, y_j, z_k) \le F_\mu(x_i, y_j, z_3) = {(i + 1)(j + 1) \over 16} = {4 \over 16}$.
	\end{proof}


\begin{thebibliography}{10}

\bibitem{AKMC}
Ahmad~N., Kim~H.-K., McCann~R.J.,
Extremal doubly stochastic measures and
optimal transportation. Bull. Math.
Sci. 1, 1,  (2011),  13--32.

\bibitem{BeigJuill}
Beiglb{\"o}ck~M., Juillet~N., 
On a problem of optimal transport under marginal martingale constraints, 
Ann. Probab. 44(1),  (2016), 42--106.


\bibitem{BHLP}
Beiglb{\"o}ck~M., Henry-Labord\`ere~P.,  Penkner~F., 
Model independent
bounds for option prices--A mass transport approach. Finance
Stoch. 17(3), (2013), 477--501.

\bibitem{BS} 
 Bene\v s~V., \v St\v ep\'an~J., 
 The support of extremal probability measures
with given marginals. In M.L. Puri, P. Revesz and W. Wertz, editors,
Mathematical Statistics and Probability Theory, (A), (1987), 33--41.

\bibitem{BoKo}
Bogachev~V.I., Kolesnikov~A.V., 
The Monge--Kantorovich problem: achievements, connections, and perspectives,
Russian Math. Surveys, 67 (2012), N 5, 785--890.


\bibitem{CLN}
Cui~L.-B., Li~W., Ng~M.K.,
Birkhoff--von Neumann Theorem for Multistochastic Tensors,
SIAM J. Matrix Anal.  Appl., 35(3), (2014), 956--973. 

\bibitem{FK}
Fraenkel~A., Kontorovich~A., The Sierpi\'nski sieve of nim-varieties and binomial
coefficients, Combinatorial Number Theory: Proceedings of the 'Integers Conference 2005' in Celebration of the 70th Birthday of Ronald Graham, Carrollton, Georgia, USA, October 27--30, 2005.

\bibitem{Galichon}
Galichon A.,  Optimal Transport Methods in Economics.  Princeton U. Press. 2016.

\bibitem{GHLT}
Galishon~A., Henry-Labord\`ere~M., Touzi~N., 
A stochastic control approach to no-arbitrage bounds given marginals, with an application to lookback options
Ann. Appl. Probab.
24(1), (2014), 312--336.


\bibitem{GhM}
Ghoussoub~N., Moameni~A., 
Symmetric Monge--Kantorovich problems and polar decompositions of vector fields,
Geometric and Functional Analysis. 24(4), (2014), 1129--1166

\bibitem{Gib}
Gibbons~K., The Geometry of Nim.  https://arxiv.org/pdf/1109.6712.pdf.

\bibitem{McCannGuill}
Guillen~N., McCann~R.J.,
In Analysis and Geometry of Metric Measure Spaces: Lecture Notes of the Seminaire de Mathematiques Superieure (SMS) Montreal 2011. G. Dafni et al, eds. Providence: Amer. Math. Soc. 145–180 (2013).


\bibitem{HL}
Henry-Labord\`ere~P., 
Model-free Hedging: A Martingale Optimal Transport Viewpoint, 
Chapman and Hall/CRC Financial Mathematics Series, 2017.

\bibitem{HW}
 Hestir~K., Williams~S.C., Supports of doubly stochastic measures.
Bernoulli 1 (1995), 217--243.

\bibitem{Kell}
Kellerer, H. G.,  Duality theorems for marginal problems. Z. Wahrsch.
Verw. Gebiete 67, (1984) 399--432. 


\bibitem{Kil}
Kiloran~N., 
Supports of extremal doubly and triply
stochastic measures - master's project. 2007.
http://www.math.toronto.edu/mccann/papers/Killoran.pdf

\bibitem{KZ1}
Kolesnikov A.V., Zaev D., Exchangeable optimal transportation and log-concavity, Theory of Stochastic Processes.  (2015) 20(2),  54--62.

\bibitem{KZ2}
Kolesnikov A.V., Zaev D.A., Optimal transportation of processes with infinite Kantorovich distance. Independence and symmetry. Kyoto J. Math., 57(2), (2017), 293--324.

\bibitem{KorMC}
McCann~R.J., Korman~J., 
Optimal transportation with capacity constraints,  Trans. Amer. Math. Soc. 367 (2015) 1501--1521.

\bibitem{KorMCS}
McCann~R.J., Korman~J., Seis~C., 
An elementary approach to linear programming duality with application to capacity constrained transport,  J. Convex Anal. 22 (2015) 797--808.



\bibitem{Lev}
Levin~V. L., The problem of mass transfer in a topological space and
probability measures with given marginal measures on the product of two
spaces. Dokl. Akad. Nauk SSSR 276, 5 (1984), 1059--1064.


\bibitem{LL}
Linial~N., Luria~Z., On the vertices of the $d$-dimensional Birkhoff polytope.
 Discrete Comput. Geom. 51 (2014), 1, 161--170.

\bibitem{M}
Mandelbrot~B., The fractal geometry of nature. W. H. Freeman and Company. 1982.

\bibitem{Moameni} 
Moameni~A., Invariance properties of the Monge-Kantorovich mass transport problem,
Discrete Continuous Dynamical Systems - A,  36 (5), (2016) 2653--2671.

\bibitem{Pass}
Pass~B., 
Multi-marginal optimal transport: Theory and applications.
ESAIM: M2AN 49 (2015) 1771--1790.

\bibitem{RR}
Rachev S.T., R{\"u}schendorf L.,
 Mass transportation problems. V.~I,~II. Springer, New York, 1998.

\bibitem{S}
Sudakov~V.N., Geometric problems in the theory of infinite dimensional probability
distributions (in Russian), Trudy Mat. Inst. Steklov, 141 (1976).

\bibitem{Vershik}
Vershik~A.M., What does a typical Markov operator look like? Algebra i Analiz, 17:5 (2005), 91--104; St. Petersburg Math. J., 17:5 (2006), 763--772. 

\bibitem{Villani}
Villani~C.,
 Topics in optimal transportation,
Amer. Math. Soc. Providence, Rhode Island, 2003.

\bibitem{Villani2}
 Villani~C.,  Optimal transport, old and new. Springer, New York, 2009.

\bibitem{Zaev}
Zaev~D.A., On the Monge--Kantorovich problem with additional linear constraints, Mathematical
Notes,  98(5), (2015) 725--741.
\end{thebibliography}
\end{document}